\documentclass[12pt,a4paper]{article}
\usepackage{latexsym, amssymb, amsthm}
\usepackage{dsfont}
\usepackage{typearea} 
\typearea{15} 

\usepackage{graphicx}
\usepackage{tikz}
\usepackage{booktabs}
\usepackage{multirow}
\usetikzlibrary{fit} 
\usepackage{pgfplots} 
\pgfplotsset{compat=1.15}
\usepackage{float}
\usepackage{hyperref}
\hypersetup{hidelinks}

\newcommand\addvmargin[1]{
  \node[fit=(current bounding box),inner ysep=#1,inner xsep=0]{};
}
\tikzset{
  topline/.style={baseline={([yshift=#1]current bounding box.north)}},
  topline/.default=-\ht\strutbox
}
\newcommand{\tentgenerator}[1]{%
\begin{tabular}{@{}c@{}}
\begin{tikzpicture}[x=\unitlength,y=\unitlength,scale=0.23,topline]
    \begin{axis}[
        xtick={-1, 0, 1},
        ytick={-1, 0, 1},
        every major tick/.append style={thick, major tick length=10pt, gray},
        axis lines=center,
        axis line style={shorten >=-10pt, shorten <=-10pt,-{>},thick,gray},
        domain=-1:1,
        line width=2pt,
        unit vector ratio*=1 1 1,
    ]
    \addplot+[mark=none, black] coordinates{#1};
    \end{axis}
    \addvmargin{5pt}
\end{tikzpicture}
\end{tabular}%
} 

\newcommand{\tentone}{\tentgenerator{(-1, 0) (0, 1) (1, 0)}}
\newcommand{\tenttwo}{\tentgenerator{(-1, 0) (-0.5, 1) (0, 0) (0.5,1) (1,0)}}
\newcommand{\tentleft}{\tentgenerator{(-1, 0) (-0.5, 1) (0, 0) (1,0)}}
\newcommand{\tentright}{\tentgenerator{(-1, 0) (0, 0) (0.5, 1) (1,0)}}
\newcommand{\tentmid}{\tentgenerator{(-1, 0) (-0.5, 0) (0, 1) (0.5,0) (1,0)}}

\DeclareMathAlphabet\gothic{U}{euf}{m}{n}

\DeclareFontFamily{U}{mathx}{}
\DeclareFontShape{U}{mathx}{m}{n}{<-> mathx10}{}
\DeclareSymbolFont{mathx}{U}{mathx}{m}{n}
\DeclareMathAccent{\widehat}{0}{mathx}{"70}
\DeclareMathAccent{\widecheck}{0}{mathx}{"71}

\makeatletter
\def\eqnarray{\stepcounter{equation}\let\@currentlabel=\theequation
\global\@eqnswtrue
\tabskip\@centering\let\\=\@eqncr
$$\halign to \displaywidth\bgroup\hfil\global\@eqcnt\z@
  $\displaystyle\tabskip\z@{##}$&\global\@eqcnt\@ne
  \hfil$\displaystyle{{}##{}}$\hfil
  &\global\@eqcnt\tw@ $\displaystyle{##}$\hfil
  \tabskip\@centering&\llap{##}\tabskip\z@\cr}

\def\endeqnarray{\@@eqncr\egroup
      \global\advance\c@equation\m@ne$$\global\@ignoretrue}

\def\@yeqncr{\@ifnextchar [{\@xeqncr}{\@xeqncr[5pt]}}
\makeatother

\begin{document}

\newtheorem{lemma}{Lemma}[section]
\newtheorem{thm}[lemma]{Theorem}
\newtheorem{cor}[lemma]{Corollary}
\newtheorem{prop}[lemma]{Proposition}

\theoremstyle{definition}

\newtheorem{remark}[lemma]{Remark}
\newtheorem{exam}[lemma]{Example}
\newtheorem{definition}[lemma]{Definition}

\newcommand{\gota}{\gothic{a}}
\newcommand{\gotb}{\gothic{b}}

\newcounter{teller}
\renewcommand{\theteller}{(\alph{teller})}
\newenvironment{tabel}{\begin{list}%
{\rm  (\alph{teller})\hfill}{\usecounter{teller} \leftmargin=1.1cm
\labelwidth=1.1cm \labelsep=0cm \parsep=0cm}
                      }{\end{list}}

\newcounter{tellerr}
\renewcommand{\thetellerr}{(\roman{tellerr})}
\newenvironment{tabeleq}{\begin{list}%
{\rm  (\roman{tellerr})\hfill}{\usecounter{tellerr} \leftmargin=1.1cm
\labelwidth=1.1cm \labelsep=0cm \parsep=0cm}
                         }{\end{list}}

\newcounter{tellerrr}
\renewcommand{\thetellerrr}{(\Roman{tellerrr})}
\newenvironment{tabelR}{\begin{list}%
{\rm  (\Roman{tellerrr})\hfill}{\usecounter{tellerrr} \leftmargin=1.1cm
\labelwidth=1.1cm \labelsep=0cm \parsep=0cm}
                         }{\end{list}}

\newcommand{\Ni}{\mathds{N}}
\newcommand{\Ri}{\mathds{R}}
\newcommand{\Ci}{\mathds{C}}

\renewcommand{\proofname}{{\bf Proof}}

\newcommand{\RRe}{\mathop{\rm Re}}
\newcommand{\IIm}{\mathop{\rm Im}}
\newcommand{\Tr}{{\mathop{\rm Tr \,}}}
\newcommand{\supp}{\mathop{\rm supp}}
\newcommand{\loc}{{\rm loc}}
\newcommand{\spann}{\mathop{\rm span}}
\newcommand{\one}{\mathds{1}}
 
\hyphenation{groups}
\hyphenation{unitary}
 
\newcommand{\tfrac}[2]{{\textstyle \frac{#1}{#2}}}

\newcommand{\cb}{{\cal B}}
\newcommand{\cl}{{\cal L}}
\newcommand{\cm}{{\cal M}}
\newcommand{\cu}{{\cal U}}

\thispagestyle{empty}

\vspace*{1cm}
\begin{center}
{\Large\bf Elliptic systems generating a positive semigroup  \\[2mm]
 are decoupled} \\[4mm]
\large Wolfgang Arendt$^1$, A.F.M. ter Elst$^2$ and Manfred Sauter$^1$  \\[8mm]
Dedicated to J\"urgen Voigt

\end{center}

\vspace{4mm}

\begin{center}
{\bf Abstract}
\end{center}

\begin{list}{}{\leftmargin=1.8cm \rightmargin=1.8cm \listparindent=10mm 
   \parsep=0pt}
\item
We show that the semigroup associated to a second-order 
elliptic system is positive if and only if the differential equations are  
essentially decoupled and the coefficients are real-valued.
This means the system can be replaced by an equivalent decoupled system
(which defines the same semigroup).
Instead of matrix coefficients we more generally consider operator-valued
coefficients, which leads to interesting additional difficulties
with respect to Banach space valued integration.

\end{list}

\vspace{6mm}
\noindent
September 2025.

\vspace{3mm}

\vspace{3mm}

\vspace{6mm}

\noindent
{\bf Home institutions:}    \\[3mm]
\begin{tabular}{@{}cl@{\hspace{10mm}}cl}
1. & Institute of Applied Analysis & 
2. & Department of Mathematics  \\
& Ulm University & 
  & University of Auckland  \\
& Helmholtzstr.\ 18  &
  & Private bag 92019 \\
& 89081 Ulm  &
   & Auckland 1142 \\ 
& Germany  &
  & New Zealand \\[8mm]
\end{tabular}

\newpage

\section{Introduction} \label{Sposvec1}

It is well known that the semigroup generated by (minus) a
second-order elliptic
differential operator with real measurable coefficients is positive 
in the sense that it maps positive functions into positive functions.
In this paper we wish to investigate what happens 
in the purely second-order setting if one replaces
the scalar-valued coefficients by coefficients with 
values in $\cl(L_2(Y))$.
If these operator-valued coefficients are multiplication operators
with real-valued functions, 
then the corresponding semigroup is again positive.
Here we use the following terminology.
Let $\Omega \subset \Ri^d$ be an open set.
By $L_2(\Omega,L_2(Y))_+$ we denote the cone of all 
$u \in L_2(\Omega,L_2(Y))$ such that $u(x) \in L_2(Y)_+$ for 
almost every $x \in \Omega$, where
$L_2(Y)_+ = \{ f \in L_2(Y) : f(y) \in [0,\infty) \mbox{ for almost every }
y \in Y \} $.
A linear map $T \colon L_2(\Omega,L_2(Y)) \to L_2(\Omega,L_2(Y))$
is called {\bf positive} if $T L_2(\Omega,L_2(Y))_+ \subset L_2(\Omega,L_2(Y))_+$.
A semigroup $S = (S_t)_{t > 0}$ on $L_2(\Omega,L_2(Y))$ is called {\bf positive}
if $S_t$ is positive for all $t > 0$.
The main aim of this paper is to prove the following.

\begin{thm} \label{tposvec101}
Let $\Omega \subset \Ri^d$ be open and 
$(Y,\Sigma,\nu)$ be a $\sigma$-finite measure space such 
that $L_2(Y)$ is separable.
For all $k,l \in \{ 1,\ldots,d \} $ let $C_{kl} \colon \Omega \to \cl(L_2(Y))$
be a bounded function such that 
$x \mapsto (C_{kl}(x) f, g)_{L_2(Y)}$ is a bounded measurable function
for all $f,g \in L_2(Y)$.
Suppose that 
\begin{equation}
\{ C_{kl}(x) : x \in \Omega \} \mbox{ is a separable subset of } \cl(L_2(Y)).
\label{etposvec101;3}
\end{equation}
Further suppose there is a $\mu > 0$ such that 
\[
\RRe \sum_{k,l=1}^d (C_{kl}(x) f_l, f_k)_{L_2(Y)}
\geq \mu \sum_{k=1}^d \|f_k\|_{L_2(Y)}^2
\]
for all $f \in L_2(Y)^d$ and a.e.\ $x \in \Omega$.
Let $V = H^1_0(\Omega,L_2(Y))$ or $V = H^1(\Omega,L_2(Y))$.
Define $\gota \colon V \times V \to \Ci$ by
\[
\gota(u,v)
= \sum_{k,l=1}^d \int_\Omega 
    (C_{kl}(x) \, (\partial_l u)(x), (\partial_k v)(x))_{L_2(Y)} \, dx
.  \]
Let $A$ be the operator in $L_2(\Omega,L_2(Y))$
associated with the form $\gota$ and $S$ the semigroup
generated by $-A$.
Then the following are equivalent.
\begin{tabeleq} 
\item \label{tposvec101-1}
The semigroup $S$ is positive.
\item \label{tposvec101-2}
For all $k,l \in \{ 1,\ldots,d \} $ there exists a bounded measurable function
$c_{kl} \colon \Omega \times Y \to \Ri$ such that
\[
\gota(u,v)
= \sum_{k,l=1}^d \int_\Omega \int_Y
      c_{kl}(x,y) \, ((\partial_l u)(x))(y) \, \overline{((\partial_k v)(x))(y)} \, dy \, dx
\]
for all $u,v \in V$.
\end{tabeleq}
\end{thm}

In the case where $Y = \{ 1,\ldots,m \} $,
equipped with the counting measure, we may identify $L_2(Y) = \Ci^m$.
Then $C_{kl}(x)$ is an $m \times m$-matrix.
The parabolic equation governed by the semigroup $S$ is a 
system of the form
\[
\dot{u} + A \, u = 0,
\quad
u(0) = u_0
,  \]
where 
\[
A v = - \sum_{k,l=1}^d \partial_k \, C_{kl} \, \partial_l v
, \quad
v = (v_1,\ldots,v_m)^{\rm T} \in D(A)
\]
and $-A$ is the generator of $S$.
Then Theorem~\ref{tposvec101} says that $S$ is positive if and only if 
there exist $c_{kl}^j \in L_\infty(\Omega,\Ri)$ such that 
\[
(A v)_j
= - \sum_{k,l=1}^d \partial_k \, (c_{kl}^j \, \partial_l v_j)
\]
for all $j \in \{ 1,\ldots,m \} $.
Thus the system decouples and the coefficients can all be chosen real.

We will see in Proposition~\ref{pposvec420}
that the implication `\ref{tposvec101-2}$\Rightarrow$\ref{tposvec101-1}'
is an easy consequence of the 
invariance criteria for convex closed subsets of Ouhabaz 
and the proof is an obvious 
modification of the proof of \cite{Ouh5} Theorem~4.2.
The main point in this paper is to prove the reverse implication.
For that we will use a projection from $\cl(L_2(Y))$
onto the space of all multiplication operators on 
$L_2(Y)$, whose existence is proved in the beautiful
paper \cite{Voi2} by J\"urgen Voigt.

The separability assumption (\ref{etposvec101;3})
in Theorem~\ref{tposvec101} implies that the 
functions $C_{kl} \colon \Omega \to \cl(L_2(Y))$ are Bochner integrable.
If $Y \subset \Ni$ with $\nu$ the counting measure, 
then we can omit this separability assumption~(\ref{etposvec101;3}).
Moreover, the decoupling becomes even more transparent. 
In fact, in that case 
\begin{eqnarray*}
L_2(\Omega, L_2(Y))
& = & \ell_2(Y,L_2(\Omega))  \\
& = & \{ (u_n)_{n \in Y} : u_n \in L_2(\Omega) 
   \mbox{ for all } n \in Y \mbox{ and } 
   \sum_{n \in Y} \|u_n\|_{L_2(\Omega)}^2 < \infty \} 
\end{eqnarray*}
and the following holds.

\begin{thm} \label{tposvec110}
Adopt the notation and assumptions of Theorem~\ref{tposvec101},
except the assumption~{\rm (\ref{etposvec101;3})}.
Suppose that $Y \subset \Ni$ and that $\nu$ is the counting measure.
Then the following are equivalent.
\begin{tabeleq}
\item \label{tposvec110-1}
The semigroup $S$ is positive.
\item \label{tposvec110-2}
For all $n \in Y$ there exists a positive contraction 
semigroup $S^{(n)}$ on $L_2(\Omega)$
such that 
\[
S_t u 
= ( S^{(n)}_t u_n )_{n \in Y}
\]
for all $u = (u_n)_{n \in Y} \in \ell_2(Y,L_2(\Omega))$
and $t > 0$.
\end{tabeleq}
\end{thm}

We would like to stress that in Theorem~\ref{tposvec101} the 
coefficients $c_{kl}$ in Property~\ref{tposvec101-2} are not unique 
in general.
See Section~\ref{Sposvec3} and in particular 
Example~\ref{xposvec200.95} for more details.

The main difficulty in the proof of the implication 
`\ref{tposvec101-1}$\Rightarrow$\ref{tposvec101-2}'
in Theorem~\ref{tposvec101}
is to show that $C_{kl}(x) + C_{lk}(x)$ is a real-valued multiplication
operator on $L_2(Y)$ 
for all $k,l \in \{ 1,\ldots,d \} $ and almost every $x \in \Omega$.
In general the symmetric part 
is needed and the operators 
$C_{kl}(x)$ are not multiplication operators if $k \neq l$.
They even do not map $L_2(Y,\Ri)$ into real-valued functions
in general.
A simple counterexample for $\Omega = \Ri^2$ is as follows.

\begin{exam} \label{xposvec103}
Let $\Omega = \Ri^2$ and $Y = \{ 1,2 \} $ with counting measure.
Choose $V = H^1(\Omega,L_2(Y))$.
For all $k,l \in \{ 1,2 \} $ define $C_{kl} \colon \Omega \to \cl(L_2(Y))$
by 
\begin{eqnarray*}
C_{11}(x) & = & C_{22}(x) = 6 I   \\[5pt]
\Big( C_{12}(x) f \Big) (1) & = & 0   \\[5pt]
\Big( C_{12}(x) f \Big) (2) & = & (3 + 4i) \, f(1)   \\[5pt]
C_{21}(x) & = & - C_{12}(x) 
\end{eqnarray*}
for all $x \in \Omega$ and $f \in L_2(Y)$.
Then $C_{12}(x)$ is not a multiplication operator
and it does not leave the space $L_2(Y,\Ri)$ invariant.
Nevertheless, one deduces easily that the ellipticity condition 
is satisfied.
Further, using Lemma~\ref{lposvec203} and \cite{Ouh5} Theorem~4.2
one quickly obtains that the associated semigroup is positive.

If one replaces the factor $(3 + 4i)$ by $1$, then the associated semigroup
is again positive, the operator $C_{12}(x)$ 
leaves the space $L_2(Y,\Ri)$ invariant, but still
$C_{12}(x)$ is not a multiplication operator.
\end{exam}

We wish to study whether the $C_{kl}(x)$ are multiplication operators,
or not.
In the preliminary Section~\ref{Sposvec2} we collect several characterisations
of multiplication operators on $L_2(Y)$ and how they can be 
combined to obtain a multiplication operator
on $L_2(\Omega \times Y) \simeq L_2(\Omega,L_2(Y))$.
The form $\gota$ does not uniquely determine
the coefficients, but it fixes
some relations between the coefficient operators.
We determine these relations in Section~\ref{Sposvec3}.
For the proof we need suitable $H^1_0(\Omega)$ functions, 
which we construct in Section~\ref{Sposvec2}.
Under a weak measurability condition on the functions
$C_{kl} \colon \Omega \to \cl(L_2(Y))$ we prove in Section~\ref{Sposvec4}
that the operators $C_{kl}(x) + C_{lk}(x)$ are multiplication 
operators on $L_2(Y)$ for almost every $x \in \Omega$, if the semigroup~$S$
is positive.
In Section~\ref{Sposvec5} we prove the implication
`\ref{tposvec101-1}$\Rightarrow$\ref{tposvec101-2}'
of Theorem~\ref{tposvec101} and many extensions.
The atomic case $L_2(Y) = \ell_2$ is discussed in Section~\ref{Sposvec6}.

Systems of coupled parabolic equations are a classical subject of 
partial differential equations with many applications.
See \cite{AngiuliLorenziMangino2}, 
\cite{AddonaLeoneLorenziRhandi}
and \cite{AngiuliLorenziMangino}
for some recent results.
For a related result to ours, namely invariance of the positive cone,
but with constant second-order coefficients on a half-space,
we refer to \cite{KresinMazya} Theorem~3.
Also the case of operator-valued coefficients has been 
considered in the 
literature, see for example 
Amann \cite{Ama9}, Denk--Hieber--Pr\"uss \cite{DHP} or 
\cite{ABK} Section~6.

\section{Preliminaries} \label{Sposvec2}

In this section we collect several results about multiplication operators,
$L_2(Y)$-valued functions and $H^1_0$-functions that we will use throughout
this paper.

Let $(Y,\Sigma,\nu)$ be a $\sigma$-finite measure space such 
that $L_2(Y)$ is separable.
We consider in general the vector spaces over the complex scalar field,
unless explicitly specified otherwise.

We need various characterisations for an operator on $L_2(Y)$
to be a multiplication operator.
The major step is originally due to Zaanen \cite{Zaa2} Theorem~7, where the 
equivalence of \ref{tposvec201-1} and \ref{tposvec201-2} in 
the following theorem was established for $L_p$ with $p \in (0,\infty)$.

\begin{thm} \label{tposvec201}
Let $Q \colon L_2(Y) \to L_2(Y)$ be a bounded operator.
Then the following are equivalent.
\begin{tabeleq}
\item \label{tposvec201-1}
$Q$ is a multiplication operator.
\item \label{tposvec201-2}
If $f,g \in L_2(Y)$ and $f \, g = 0$ a.e., then $(Qf) \, g = 0$ a.e.
\item \label{tposvec201-3}
$\nu([Qf \neq 0] \cap [f=0]) = 0$ for all $f \in L_2(Y)$.
\item \label{tposvec201-4}
$Q (\one_A \, f) = \one_A \, Qf$ for all $f \in L_2(Y)$ and measurable $A \subset Y$.
\item \label{tposvec201-5}
There exists a $c > 0$ such that $|Q f| \leq c \, |f|$ for all $f \in L_2(Y)$.
\item \label{tposvec201-6}
The operator $Q$ is the linear combination of bounded operators $T$ such that
both $I - T$ and $I + T$ are positive.
\end{tabeleq}
\end{thm}
\begin{proof}
The implications 
\ref{tposvec201-1}$\Rightarrow$\ref{tposvec201-2}$\Rightarrow$\ref{tposvec201-3} 
are trivial.
For the implication \ref{tposvec201-3}$\Rightarrow$\ref{tposvec201-4}
note that $\nu([Q (\one_A \, f) \neq 0] \cap A^{\rm c})
\leq \nu([Q (\one_A \, f) \neq 0] \cap [\one_A \, f = 0]) = 0$.
Hence $Q (\one_A \, f) = \one_A \, Q (\one_A \, f)$
and $\one_A \, Q f = \one_A \, Q (\one_A \, f) + \one_A \, Q (\one_{A^{\rm c}} \, f)
= \one_A \, Q (\one_A \, f) = Q (\one_A \, f)$.
A short proof for the implication \ref{tposvec201-4}$\Rightarrow$\ref{tposvec201-1}
is in \cite{ArT} Proposition~1.7.

The implications 
\ref{tposvec201-1}$\Rightarrow$\ref{tposvec201-5}$\Rightarrow$\ref{tposvec201-3} 
and \ref{tposvec201-1}$\Rightarrow$\ref{tposvec201-6}
are trivial.
Finally we prove \ref{tposvec201-6}$\Rightarrow$\ref{tposvec201-5}.
Let $T$ be a bounded operator such that
both $I - T$ and $I + T$ are positive.
Let $u \in L_2(Y)$ with $u \geq 0$.
Then $u - T u \geq 0$ and $u + T u \geq 0$.
Hence $|T u| \leq |u|$.
Therefore if $u \in L_2(Y)$, then $|T u| \leq 4 |u|$ and \ref{tposvec201-5}
is valid.
\end{proof}

\begin{cor} \label{cposvec202}
Let $Q \colon L_2(Y) \to L_2(Y)$ be a bounded operator.
Suppose $Q$ is not a multiplication operator.
Then there exist $f \in L_2(Y,\Ri)$ and a measurable set $B \subset Y$
such that $f \geq 0$, $B \subset [f = 0]$, $0 < \nu(B) < \infty$
and $(Q f, \one_B)_{L_2(Y)} \neq 0$.
\end{cor}
\begin{proof}
By Theorem~\ref{tposvec201}\ref{tposvec201-3}$\Rightarrow$\ref{tposvec201-1} 
there exists an $f \in L_2(Y)$
such that the measure $\nu([Qf \neq 0] \cap [f=0]) > 0$.
Because $[Qf \neq 0] \subset \bigcup_{k=0}^3 [Q ((\RRe( i^k f))^+) \neq 0]$
and $[f=0] = \bigcap_{k=0}^3 [(\RRe (i^k f))^+ = 0]$
it follows that 
\[
0 < \nu([Qf \neq 0] \cap [f=0]) 
\leq \sum_{k=0}^3 \nu([Q ((\RRe( i^k f))^+) \neq 0] \cap [(\RRe (i^k f))^+ = 0])
.  \]
Hence there exists a $g \in L_2(Y,\Ri)$ such that 
$g \geq 0$ and 
$\nu([Qg \neq 0] \cap [g=0]) > 0$.
Finally $[Qg \neq 0] \subset \bigcup_{k=0}^3 [\RRe (i^k Qg) > 0]$.
Consequently there is a $k \in \{ 0,1,2,3 \} $ such that 
$\nu([\RRe (i^k Qg) > 0] \cap [g=0]) > 0$.
Since $(Y,\Sigma,\nu)$ is $\sigma$-finite, there exists a measurable 
$B \subset [\RRe (i^k Qg) > 0] \cap [g=0]$ such that $0 < \nu(B) < \infty$.
Then $\RRe i^k (Qg, \one_B)_{L_2(Y)} > 0$.
\end{proof}

Next fix an open set $\Omega \subset \Ri^d$.
In a natural way the spaces $L_2(\Omega, L_2(Y))$
and $L_2(\Omega \times Y)$ are unitarily equivalent.
Explicitly, define $\Phi \colon L_2(\Omega \times Y) \to L_2(\Omega, L_2(Y))$
by 
\begin{equation}
( \Phi(\alpha) ) (x) = \alpha(x,\cdot)
\label{epposvec230;1}
\end{equation}
Then $\Phi$ is a unitary map by \cite{HNVW1} Proposition~1.1.24.
Given a map $C \colon \Omega \to \cl(L_2(Y))$
such that 
$x \mapsto (C(x) f, g)_{L_2(Y)}$ is a bounded measurable function
for all $f,g \in L_2(Y)$,
one can use \cite{HNVW1} Proposition~1.1.28 to
define a map $T \colon L_2(\Omega,L_2(Y)) \to L_2(\Omega,L_2(Y))$
by $(T u)(x) = C(x) \, u(x)$.
Using the unitary map one associates to $T$ a map 
$\widetilde T \colon L_2(\Omega \times Y) \to L_2(\Omega \times Y)$.
We characterise when $\widetilde T$ is a multiplication operator in terms 
of when the $C(x)$ are multiplication operators.

\begin{prop} \label{pposvec230}
Let $C \colon \Omega \to \cl(L_2(Y))$ be a function such that 
$x \mapsto (C(x) f, g)_{L_2(Y)}$ is a bounded measurable function
for all $f,g \in L_2(Y)$.
Define $T \colon L_2(\Omega,L_2(Y)) \to L_2(\Omega,L_2(Y))$
by $(T u)(x) = C(x) \, u(x)$.
Define $\Phi \colon L_2(\Omega \times Y) \to L_2(\Omega, L_2(Y))$
as in {\rm (\ref{epposvec230;1})} and define 
$\widetilde T \colon L_2(\Omega \times Y) \to L_2(\Omega \times Y)$
by $\widetilde T = \Phi^{-1} \, T \, \Phi$.
Then the following are equivalent.
\begin{tabeleq}
\item \label{pposvec230-1}
For almost every $x \in \Omega$ the map $C(x)$ is a multiplication 
operator on $L_2(Y)$ [with a real-valued function].
\item \label{pposvec230-2}
$\widetilde T$ is a multiplication operator on $L_2(\Omega \times Y)$
 [with a real-valued function].
\end{tabeleq}
\end{prop}
\begin{proof}
`\ref{pposvec230-1}$\Rightarrow$\ref{pposvec230-2}'.
If $\alpha \colon \Omega \times Y \to \Ci$ is a  
function and $x \in \Omega$, then define $\alpha_x = \alpha(x,\cdot) \colon Y \to \Ci$.
If $\alpha \in L_2(\Omega \times Y)$, then $\alpha_x \in L_2(Y)$
for almost every $x \in \Omega$.
Note that $(\widetilde T \alpha)_x = C(x) (\alpha_x)$ for almost every $x \in \Omega$.

Now let $\alpha,\beta \in L_2(\Omega \times Y)$ and suppose that
$\alpha \, \beta = 0$ almost everywhere on $\Omega \times Y$.
Then by Fubini's theorem $\alpha_x \, \beta_x = 0$ a.e.\ on $Y$ 
for almost every $x \in \Omega$.
Since $C(x)$ is a multiplication operator on $L_2(Y)$ for 
a.e.\ $x \in \Omega$, one obtains 
\[
(\widetilde T \alpha)_x \, \beta_x
= (C(x) \alpha_x) \, \beta_x 
= 0
\quad \mbox{a.e.\ on } Y
\]
for a.e.\ $x \in \Omega$.
Because $(\widetilde T \alpha) \, \beta$ is measurable 
from $\Omega \times Y$ into $\Ci$ one can use again Fubini's theorem 
to deduce that $(\widetilde T \alpha) \, \beta = 0$ a.e.\ on $\Omega \times Y$.
Then $\widetilde T$ is a multiplication operator by 
Theorem~\ref{tposvec201}\ref{tposvec201-2}$\Rightarrow$\ref{tposvec201-1}.

If $C(x)$ is a multiplication with a real-valued function for almost 
every $x \in \Omega$, then $C(x)$ leaves $L_2(Y,\Ri)$ invariant.
Let $\alpha,\beta \in L_2(\Omega \times Y,\Ri)$.
Then Fubini implies that 
$(\widetilde T \alpha, \beta)_{L_2(\Omega \times Y)}
= \int_\Omega ( (\widetilde T \alpha)_x, \beta_x)_{L_2(Y)} \in \Ri$.
Hence $\widetilde T$ leaves $L_2(\Omega \times Y,\Ri)$ invariant
and consequently it is a multiplication operator with a real-valued function.

The implication \ref{pposvec230-2}$\Rightarrow$\ref{pposvec230-1}
is easy.
\end{proof}

For the proof of Theorem~\ref{tposvec101} we make use of 
the following commutation property of partial derivatives
of $H^1_0(\Omega)$ functions.

\begin{lemma} \label{lposvec203}
Let $k,l \in \{ 1,\ldots,d \} $ and $\varphi,\psi \in H^1_0(\Omega)$.
Then 
$\int_\Omega (\partial_k \varphi) \, \partial_l \psi 
= \int_\Omega (\partial_l \varphi) \, \partial_k \psi$.
\end{lemma}
\begin{proof}
If $\varphi,\psi \in C_c^\infty(\Omega)$, then 
$(\partial_k \varphi, \partial_l \psi)_{L_2(\Omega)}
= - (\varphi, \partial_k \partial_l \psi)_{L_2(\Omega)}
= - (\varphi, \partial_l \partial_k \psi)_{L_2(\Omega)}
= (\partial_l \varphi, \partial_k \psi)_{L_2(\Omega)}$.
Then by density 
$(\partial_k \varphi, \partial_l \psi)_{L_2(\Omega)} 
= (\partial_l \varphi, \partial_k \psi)_{L_2(\Omega)}$
for all $\varphi,\psi \in H^1_0(\Omega)$
and the lemma follows.
\end{proof}

Next we construct test functions relative to a pair of indices.

\begin{lemma} \label{lposvec204}
Let $\tau \in \Ri$.
Let $\tilde k,\tilde l \in \{ 1,\ldots,d \} $.
Then there exist $\varphi,\psi \in W^{1,\infty}_0((-1,1)^d) \subset W^{1,\infty}(\Ri^d)$ 
such that 
$\varphi \geq 0$, $\psi \geq 0$, 
\[
\int_{\Ri^d} (\partial_{\tilde l} \varphi) \, \partial_{\tilde k} \psi 
= \int_{\Ri^d} (\partial_{\tilde k} \varphi) \, \partial_{\tilde l} \psi 
= \tau
\]
and 
\begin{equation}
\int_{\Ri^d} (\partial_l \varphi) \, \partial_k \psi 
= 0 
\label{elposvec204;1}
\end{equation}
for all $k,l \in \{ 1,\ldots,d \} $ with 
$(k,l) \not\in \{ (\tilde k,\tilde l), (\tilde l,\tilde k) \} $.
\end{lemma}
\begin{proof}
Define $\eta \colon \Ri \to \Ri$ by $\eta(t) = (1 - |t|)^+$.
Further we define $\rho \colon \Ri \to \Ri$ by 
$\rho(t) = \eta(2(t - \tfrac{1}{2})) + \eta(2(t + \tfrac{1}{2}))$.
Then $\eta,\rho \in W^{1,\infty}_0(-1,1)$
and 
\[
\int_\Ri \eta'(t) \, \rho(t) \, dt
= \int_\Ri \eta(t) \, \rho'(t) \, dt
= \int_\Ri \eta'(t) \, \rho'(t) \, dt
= 0
\mbox{ and } 
\int_\Ri \eta(t) \, \rho(t) \, dt = \frac{1}{2}
.  \]
In the following cases the functions $\varphi$ and $\psi$ 
will be obtained as products of functions that each 
depend only on a single one of the variables $x_1,\ldots,x_d$.
Hence if $k \not\in \{ \tilde k,\tilde l \} $, then Fubini's theorem
and considering the integral over the $k$-th variable
gives (\ref{elposvec204;1}) in each of the following cases.
The same applies if $l \not\in \{ \tilde k,\tilde l \} $.

We distinguish five cases.

\noindent
{\bf Case 1.} Suppose $\tau > 0$ and $\tilde k = \tilde l$.
Define $\varphi,\psi \colon \Ri^d \to \Ri$ by 
\[
\varphi(x) 
= 2^{d-2} \, \tau \, 
\prod_{k = 1}^d \eta(x_k)
\quad \mbox{and} \quad 
\psi(x) 
= \eta(x_{\tilde l}) 
\prod_{k \neq \tilde l} \rho(x_k)
.  \]

\begin{table}[H]
\centering
\begin{tabular}{ccccc}
\hline
& $x_1$ & $x_2$ & $x_3$ & \ldots  \\
$\varphi$ & \tentone & \tentone & \tentone & \ldots \\
$\psi$ & \tentone & \tenttwo & \tenttwo & \ldots \\
\hline
\end{tabular}
\caption{$\tilde{k}=\tilde{l}=1$ with $\tau > 0$ (Case 1).}
\end{table}

\noindent
{\bf Case 2.} Suppose $\tau > 0$ and $\tilde k \neq \tilde l$.
Define $\varphi,\psi \colon \Ri^d \to \Ri$ by 
\begin{eqnarray*}
\varphi(x) 
& = & 2^d \, \tau \, \prod_{k = 1}^d \eta(x_k)
\quad \mbox{and}  \\
\psi(x) 
& = & \eta(2(x_{\tilde k} + \tfrac{1}{2})) \,
\eta(2(x_{\tilde l} - \tfrac{1}{2}))
\prod_{k \not\in \{ \tilde k, \tilde l \} } \rho(x_k)
.  
\end{eqnarray*}

\begin{table}[H]
\centering
\begin{tabular}{ccccc}
\hline
& $x_1$ & $x_2$ & $x_3$ & \ldots  \\
$\varphi$ & \tentone & \tentone & \tentone & \ldots \\
$\psi$ & \tentleft & \tentright & \tenttwo & \ldots \\
\hline
\end{tabular}
\caption{$(\tilde{k},\tilde{l})=(1,2)$ with $\tau > 0$ (Case 2).}
\end{table}

\noindent
{\bf Case 3.} Suppose $\tau < 0$ and $\tilde k = \tilde l$.
Define $\varphi,\psi \colon \Ri^d \to \Ri$ by 
\[
\varphi(x) 
= 2^{d-2} \, |\tau| \, \eta(2(x_{\tilde l} + \tfrac{1}{2})) 
\prod_{k \neq \tilde l} \eta(x_k)
\quad \mbox{and} \quad 
\psi(x) 
= \eta(2 x_{\tilde l}) 
\prod_{k \neq \tilde l} \rho(x_k)
.  \]

\begin{table}[H]
\centering
\begin{tabular}{ccccc}
\hline
& $x_1$ & $x_2$ & $x_3$ & \ldots  \\
$\varphi$ & \tentleft & \tentone & \tentone & \ldots \\
$\psi$ & \tentmid & \tenttwo & \tenttwo & \ldots \\\hline
\end{tabular}
\caption{$\tilde{k}=\tilde{l}=1$ with $\tau < 0$ (Case 3).}
\end{table}

\noindent
{\bf Case 4.} Suppose $\tau < 0$ and $\tilde k \neq \tilde l$.
Define $\varphi,\psi \colon \Ri^d \to \Ri$ by 
\begin{eqnarray*}
\varphi(x) 
& = & 2^d \, |\tau| \, \prod_{k = 1}^d \eta(x_k)
\quad \mbox{and}  \\ 
\psi(x) 
& = & \eta(2(x_{\tilde k} - \tfrac{1}{2})) \,
\eta(2(x_{\tilde l} - \tfrac{1}{2}))
\prod_{k \not\in \{ \tilde k, \tilde l \} } \rho(x_k)
.  
\end{eqnarray*}

\begin{table}[H]
\centering
\begin{tabular}{ccccc}
\hline
& $x_1$ & $x_2$ & $x_3$ & \ldots  \\
$\varphi$ & \tentone & \tentone & \tentone & \ldots \\
$\psi$ & \tentright & \tentright & \tenttwo & \ldots \\\hline
\end{tabular}
\caption{$(\tilde{k},\tilde{l}) = (1,2)$ with $\tau < 0$ (Case 4).}
\end{table}

\noindent
{\bf Case 5.} Suppose $\tau = 0$.
This case is trivial, choose $\varphi = \psi = 0$.

\smallskip

Then the lemma follows after straightforward calculations.
\end{proof}

We also need $L_2(Y)$-valued Sobolev spaces.
For details we refer to \cite{HNVW1} Subsections~2.5.a and 2.5.b,
and \cite{ArendtKreuter} Corollary~4.14.
For the convenience of the reader, recall that 
$H^1(\Omega,L_2(Y))$ is the Hilbert space of all 
$u \in L_2(\Omega, L_2(Y))$ for which there are $u_1,\ldots,u_d \in L_2(\Omega, L_2(Y))$
such that 
\[
\int_\Omega (\partial_k \tau)(x) \, u(x) \, dx
= - \int_\Omega \tau(x) \, u_k(x) \, dx
\]
for all $\tau \in C_c^\infty(\Omega)$
and 
$\|u\|_{H^1(\Omega,L_2(Y))}^2 = \|u\|_{L_2(\Omega,L_2(Y))}^2 + \sum_{k=1}^d \|u_k\|_{L_2(\Omega,L_2(Y))}^2$.
Furthermore, $H^1_0(\Omega,L_2(Y))$
is the closure in $H^1(\Omega,L_2(Y))$ of all infinitely differentiable 
functions from $\Omega$ into $L_2(Y)$ with compact support.
Using the unitary map (\ref{epposvec230;1}) it follows as in the proof of Lemma~7.6 in \cite{GT}
that $u^+ \in H^1(\Omega, L_2(Y,\Ri))$ and 
$\partial_k (u^+) = \one_{[u > 0]} \, \partial_k u$ for all 
$u \in H^1(\Omega,L_2(Y,\Ri))$ and $k \in \{ 1,\ldots,d \} $.

If $\varphi \in L_2(\Omega)$ and $f \in L_2(Y)$ then define 
$\varphi \otimes f \in L_2(\Omega,L_2(Y))$ by 
$(\varphi \otimes f)(x) = \varphi(x) \, f \in L_2(Y)$.
If $\Phi \colon L_2(\Omega \times Y) \to L_2(\Omega, L_2(Y))$
is the unitary map as in (\ref{epposvec230;1}), then 
$(\Phi^{-1}( \varphi \otimes f ))(x,y) = \varphi(x) \, f(y)$
for almost every $(x,y) \in \Omega \times Y$.

One has the following density properties.

\begin{lemma} \label{lposvec213}
\mbox{}
\begin{tabel}
\item \label{lposvec213-1}
The set
$\spann \{ \varphi \otimes f : 
    \varphi \in H^1_0(\Omega) \mbox{ and } f \in L_2(Y) \} $
is dense in $H^1_0(\Omega,L_2(Y))$.
\item \label{lposvec213-2}
The set
$\spann \{ \varphi \otimes f : 
    \varphi \in H^1(\Omega) \mbox{ and } f \in L_2(Y) \} $
is dense in $H^1(\Omega,L_2(Y))$.
\end{tabel}
\end{lemma}
\begin{proof}
`\ref{lposvec213-1}'.
We only give the proof in case $L_2(Y)$ is infinite dimensional.
Let $ \{ e_n : n \in \Ni \} $ be an orthonormal basis for $L_2(Y)$.
Let $u \in H^1_0(\Omega,L_2(Y))$.
Let $n \in \Ni$.
Define $\varphi_n \colon \Omega \to \Ci$ by 
$\varphi_n(x) = (u(x), e_n)_{L_2(Y)}$.
Then $\varphi_n \in H^1_0(\Omega)$.
If $N \in \Ni$, then 
$\|\sum_{n=1}^N \varphi_n(x) \, e_n\|_{L_2(Y)}
\leq \|u(x)\|_{L_2(Y)}$ for almost every $x \in \Omega$.
Hence it follows from the Lebesgue dominated convergence theorem that 
$\lim_{N \to \infty} \sum_{n=1}^N \varphi_n \otimes e_n = u$ in $L_2(\Omega,L_2(Y))$.
Similarly 
$\lim_{N \to \infty} \partial_k \sum_{n=1}^N \varphi_n \otimes e_n = \partial_k u$ in $L_2(\Omega,L_2(Y))$
for all $k \in \{ 1,\ldots,d \} $.
Therefore $\lim_{N \to \infty} \sum_{n=1}^N \varphi_n \otimes e_n = u$ in $H^1(\Omega,L_2(Y))$
and the statement follows.

`\ref{lposvec213-2}'.
This is similar.
\end{proof}

Lemmas~\ref{lposvec203} and \ref{lposvec213}\ref{lposvec213-1} can be combined
to give another commutation relation.

\begin{lemma} \label{lposvec231}
Let $B \in \cl(L_2(Y))$.
Then 
\[
\int_\Omega ( \Big( B (\partial_l u) \Big) (x), (\partial_k v)(x) )_{L_2(Y)} \, dx
= \int_\Omega ( \Big( B (\partial_k u) \Big) (x), (\partial_l v)(x) )_{L_2(Y)}  \, dx
\]
for all $k,l \in \{ 1,\ldots,d \} $ and $u,v \in H^1_0(\Omega,L_2(Y))$.
\end{lemma}
\begin{proof}
If $u,v \in \{ \varphi \otimes f : 
    \varphi \in H^1_0(\Omega) \mbox{ and } f \in L_2(Y) \} $,
then the equality follows from Lemma~\ref{lposvec203}.
Now use continuity and the density of Lemma~\ref{lposvec213}\ref{lposvec213-1}.
\end{proof}

\section{Sesquilinear forms} \label{Sposvec3}

In this section we study sesquilinear forms in which the form domain 
is an $L_2(Y)$-valued Sobolev space and the coefficients are 
operators in $L_2(Y)$. 
We wish to infer information about the coefficients from the form.
In this section we neither need that the form is elliptic,
nor that the form domain is a Hilbert space.

Throughout this section let $\Omega \subset \Ri^d$, with $d \geq 2$, be open and 
let $(Y,\Sigma,\nu)$ be a $\sigma$-finite measure space such 
that $L_2(Y)$ is separable.
For all $k,l \in \{ 1,\ldots,d \} $ let $C_{kl} \colon \Omega \to \cl(L_2(Y))$
be a function such that for all $f,g \in L_2(Y)$
the map $x \mapsto (C_{kl}(x) f,g)_{L_2(Y)}$ is 
a bounded measurable function from $\Omega$ into $\Ci$.
Note that these assumptions imply that there exists an $M \geq 0$ 
such that $\|C_{kl}(x)\|_{\cl(L_2(Y))} \leq M$
for all $k,l \in \{ 1,\ldots,d \} $ and $x \in \Omega$ by the 
uniform boundedness principle.
In fact, for every fixed $f \in L_2(Y)$ the set 
$ \{ C_{kl}(x) f : k,l \in \{ 1,\ldots,d \} \mbox{ and } x \in \Omega \} $
is weakly bounded and therefore bounded in $L_2(Y)$.
Let $V$ be a subspace of $H^1(\Omega,L_2(Y))$ such that 
$H^1_0(\Omega,L_2(Y)) \subset V$.
Define $\gota \colon V \times V \to \Ci$ by
\[
\gota(u,v)
= \sum_{k,l=1}^d \int_\Omega 
    (C_{kl}(x) \, (\partial_l u)(x), (\partial_k v)(x))_{L_2(Y)} \, dx
.  \]
Our main interest lies in the cases $V = H^1_0(\Omega,L_2(Y))$
and $V = H^1(\Omega,L_2(Y))$, but several of the intermediate results 
are valid for a general form domain.

In order to obtain information about the $C_{kl}$ from the form $\gota$
we use the notion of Lebesgue points.
We need two lemmas.

\begin{lemma} \label{lposvec205}
Let $\varphi \colon \Omega \to \Ri$ be bounded measurable, $r > 0$ and 
$\eta \in L_\infty(\Ri^d)$ with $\supp \eta \in B(0,r)$.
Let $x_0 \in \Omega$ be a Lebesgue point for $\varphi$, that is 
$\lim_{\delta \downarrow 0} \frac{1}{|B(x_0,\delta)|} \int_{B(x_0,\delta)} |\varphi(x) - \varphi(x_0)| = 0$.
Then 
\[
\lim_{\delta \downarrow 0} \int_{B(0,r)} \eta(z) \, \varphi(x_0 + \delta z) \, dz 
= \int_{\Ri^d} \eta(z) \, \varphi(x_0) \, dz 
.  \]
\end{lemma}
\begin{proof}
There exists a $\delta_0 > 0$ such that $B(x_0,r \delta_0) \subset \Omega$.
Let $\delta \in (0,\delta_0]$.
Then 
\begin{eqnarray*}
\Big| \int_{B(0,r)} \eta(z) \, \Big( \varphi(x_0 + \delta z) - \varphi(x_0) \Big) \, dz \Big|
& \leq & \|\eta\|_\infty \, \int_{B(0,r)} | \varphi(x_0 + \delta z) - \varphi(x_0) | \, dz  \\
& = & \|\eta\|_\infty \, \frac{|B(0,r)|}{|B(x_0, r \delta)|} \, \int_{B(x_0, r \delta)} | \varphi(x) - \varphi(x_0) | \, dx 
.
\end{eqnarray*}
Since $x_0$ is a Lebesgue point for $\varphi$, the lemma follows.
\end{proof}

\begin{lemma} \label{lposvec206}
There exists a null set $N \subset \Omega$ such that for all 
$x_0 \in \Omega \setminus N$, $k,l \in \{ 1,\ldots,d \} $ and $f,g \in L_2(Y)$
the point $x_0$ is a Lebesgue point of the function $x \mapsto (C_{kl}(x) f, g)_{L_2(Y)}$.
\end{lemma} 
\begin{proof}
Since $L_2(Y)$ is separable, there exists a countable subset $H_0 \subset L_2(Y)$
such that $H_0$ is dense in $L_2(Y)$.
Since almost every point of $\Omega$ is a Lebesgue point for a given $L_{1,\loc}(\Omega)$
function, there exists a null set $N \subset \Omega$ such that for all 
$x_0 \in \Omega \setminus N$, $k,l \in \{ 1,\ldots,d \} $ and $f,g \in H_0$
the point $x_0$ is a Lebesgue point of the bounded 
function $x \mapsto (C_{kl}(x) f, g)_{L_2(Y)}$.
Because the $C_{kl}(x)$ are uniformly bounded operators, the same set $N$ works 
for all $f,g \in L_2(Y)$.
\end{proof}

In general only the coefficients $C_{kl} + C_{lk}$ are determined uniquely by the 
form $\gota$.

\begin{prop} \label{pposvec200.8}
Suppose 
\[
\gota(u,v)
= 0
\]
for all $u,v \in V$.
Then 
\[
C_{kl}(x) + C_{lk}(x) = 0
\]
for all $k,l \in \{ 1,\ldots,d \} $ and almost all $x \in \Omega$.
\end{prop}
\begin{proof}
Let the null set $N \subset \Omega$ be as in Lemma~\ref{lposvec206}.
Let $\tilde k, \tilde l \in \{ 1,\ldots,d \} $ and choose 
$\varphi,\psi \in W^{1,\infty}(\Ri^d)$ as in Lemma~\ref{lposvec204}
with $\tau = 1$ if $\tilde k \neq \tilde l$, and 
with $\tau = 2$ if $\tilde k = \tilde l$.
Then $\supp (\partial_l \varphi) \, \partial_k \psi \subset [-1,1]^d$
for all $k,l \in \{ 1,\ldots,d \} $.

Let $x_0 \in \Omega \setminus N$.
For all $\delta \in (0,\infty)$ 
define $\varphi_\delta,\psi_\delta \in W^{1,\infty}_c(\Ri^d)$
by $\varphi_\delta(x) = \varphi(\delta^{-1} \, (x-x_0))$ and
$\psi_\delta(x) = \psi(\delta^{-1} \, (x-x_0))$.
Let $f,g \in L_2(Y)$.
If $\delta > 0$ is small enough, then 
$\varphi_\delta \otimes f, \psi_\delta \otimes g \in H^1_0(\Omega,L_2(Y))$ and 
\begin{eqnarray*}
\lefteqn{
0 
= \gota(\varphi_\delta \otimes f, \psi_\delta \otimes g)
} \hspace*{5mm}  \\*
& = & \sum_{k,l=1}^d \int_\Omega (\partial_l \varphi_\delta)(x) \, (\partial_k \psi_\delta)(x) \, 
     (C_{kl}(x) f, g)_{L_2(Y)} \, dx  \\
& = & \delta^{d - 2} \sum_{k,l=1}^d \int_{B(0,d)} ((\partial_l \varphi) \, \partial_k \psi)(z) \, 
     (C_{kl}(x_0 + \delta z) f, g)_{L_2(Y)} \, dz  \\
& = & \delta^{d - 2} \Bigg( 
\sum_{k,l=1}^d \int_{\Ri^d} ((\partial_l \varphi) \, \partial_k \psi)(z) \, 
     (C_{kl}(x_0) f, g)_{L_2(Y)} \, dz  
\\*
& & \hspace*{11mm} {}
+ \sum_{k,l=1}^d \int_{B(0,d)} ((\partial_l \varphi) \, \partial_k \psi)(z) \, 
     \Big( (C_{kl}(x_0 + \delta z) f, g)_{L_2(Y)} - (C_{kl}(x_0) f, g)_{L_2(Y)} \Big) \, dz  
                    \Bigg)
.
\end{eqnarray*}
Now
\[
\lim_{\delta \downarrow 0}
 \sum_{k,l=1}^d \int_{B(0,d)} ((\partial_l \varphi) \, \partial_k \psi)(z) \, 
     \Big( (C_{kl}(x_0 + \delta z) f, g)_{L_2(Y)} - (C_{kl}(x_0) f, g)_{L_2(Y)} \Big) \, dz  
= 0
\]
by Lemmas~\ref{lposvec206} and \ref{lposvec205}.
Hence 
\begin{eqnarray*}
0
& = & \sum_{k,l=1}^d \int_{\Ri^d} ((\partial_l \varphi) \, \partial_k \psi)(z) \, 
     (C_{kl}(x_0) f, g)_{L_2(Y)} \, dz    \\
& = & (C_{\tilde k \tilde l}(x_0) f, g)_{L_2(Y)}
   + (C_{\tilde l \tilde k}(x_0) f, g)_{L_2(Y)}
\end{eqnarray*}
and the proposition follows.
\end{proof}

If $d \geq 2$ and $V = H^1_0(\Omega,L_2(Y))$ then the form $\gota$
never determines the coefficients.
In fact, By Lemma~\ref{lposvec203} one may always replace 
$C_{12}(x)$ by $C_{12}(x) + I_{L_2(Y)}$ and 
$C_{21}(x)$ by $C_{21}(x) - I_{L_2(Y)}$ without changing the form.

If $d = 2$, $\Omega$ is bounded and $V = H^1(\Omega,L_2(Y))$, then 
also the off-diagonal elements are determined by the form.

\begin{prop} \label{pposvec200.9}
Suppose $d = 2$, $\Omega$ is bounded and $V = H^1(\Omega,L_2(Y))$.
Further suppose 
\[
\gota(u,v)
= 0
\]
for all $u,v \in V$.
Then 
\[
C_{kl}(x) = 0
\]
for all $k,l \in \{ 1,2 \} $ and almost all $x \in \Omega$.
\end{prop}
\begin{proof}
It follows from Proposition~\ref{pposvec200.8} that 
$C_{11}(x) = C_{22}(x) = C_{12}(x) + C_{21}(x) = 0$ for almost every
$x \in \Omega$.
Let $f,g \in L_2(Y)$.
Let $\varphi,\psi \in H^1(\Omega)$.
Then 
\[
0
= \gota(\varphi \otimes f, \psi \otimes g)
= \int_\Omega ( C_{12}(x) f, g)_{L_2(Y)} 
          \, \Big( (\partial_2 \varphi)(x) \, \overline{(\partial_1 \psi)(x)}
   - (\partial_1 \varphi)(x) \, \overline{(\partial_2 \psi)(x)} \Big) \, dx
\]
for all $\varphi,\psi \in H^1(\Omega)$.
Let $\varphi \in H^1(\Omega)$.
Choose $\psi(x) = x_1$.
Then 
\begin{equation}
\int_\Omega ( C_{12}(x) f, g)_{L_2(Y)} \, (\partial_2 \varphi)(x) \, dx = 0 .
\label{epposvec200.9;1}
\end{equation}
Similarly $\int_\Omega ( C_{12}(x) f, g)_{L_2(Y)} \, (\partial_1 \varphi)(x) \, dx = 0$.
So $x \mapsto (C_{12}(x) f, g)_{L_2(Y)}$ is an element of 
$W^{1,\infty}(\Omega)$ with vanishing gradient.
Hence $x \mapsto (C_{12}(x) f, g)_{L_2(Y)}$ is constant a.e.\ on connected 
components of $\Omega$ by \cite{Zie2} Corollary~2.1.9.
Let $\Omega_0$ be a connected component of $\Omega$.
Define $\varphi \in H^1(\Omega)$ by $\varphi(x) = x_2$
if $x \in \Omega_0$, and $\varphi(x) = 0$ if $x \in \Omega \setminus \Omega_0$.
Then (\ref{epposvec200.9;1}) gives 
$\int_{\Omega_0} ( C_{12}(x) f, g)_{L_2(Y)} \, dx = 0$ and 
$x \mapsto (C_{12}(x) f, g)_{L_2(Y)} = 0$ a.e.\ on $\Omega_0$.
So $C_{12}(x) = 0$ for a.e.\ $x \in \Omega$.
\end{proof}

One might hope that Proposition~\ref{pposvec200.9} is also valid 
in higher dimensions, but that is not the case in general,
even if $L_2(Y) = \Ci$, by the next example.

\begin{exam} \label{xposvec200.95}
Let $\Omega = (-1,1)^3$.
For all $k,l \in \{ 1,2,3 \} $ define $c_{kl} \colon \Omega \to \Ri$ by 
\begin{eqnarray*}
c_{12}(x) & = & - (x_1^2 - 1) \, (x_2^2 - 1) \, x_3  \\
c_{23}(x) & = & - (x_2^2 - 1) \, (x_3^2 - 1) \, x_1  \\
c_{13}(x) & = & + (x_1^2 - 1) \, (x_3^2 - 1) \, x_2  
\end{eqnarray*} 
together with $c_{11} = c_{22} = c_{33} = 0$ and 
$c_{21} = - c_{12}$, $c_{32} = - c_{23}$ and $c_{31} = - c_{13}$.
Note that 
\begin{equation}
\partial_2 c_{12} = - \partial_3 c_{13}
, \quad
\partial_1 c_{12} = \partial_3 c_{23}
\mbox{ and} \quad
\partial_1 c_{13} = - \partial_2 c_{23}
.  
\label{exposvec200.95;1}
\end{equation}
Moreover $c_{kl}(x) = 0$ if $x_k \in \{ -1,1 \} $ or $x_l \in \{ -1,1 \} $
for all $k,l \in \{ 1,2,3 \} $.
Define $\gota \colon H^1(\Omega) \times H^1(\Omega) \to \Ci$ by 
\[
\gota(\varphi,\psi) 
= \sum_{k,l=1}^3 \int_\Omega c_{kl}(x) \, 
      (\partial_l \varphi)(x) \, \overline{(\partial_k \psi)(x)} \, dx
.  \]
Let $\varphi,\psi \in C^\infty(\overline \Omega)$.
Then integration by parts and using that the boundary terms vanish gives
\begin{eqnarray*}
\gota(\varphi,\psi)
& = & \sum_{k < l} \int_\Omega c_{kl}(x) \Big( 
        (\partial_l \varphi)(x) \, \overline{(\partial_k \psi)(x)}
       - (\partial_k \varphi)(x) \, \overline{(\partial_l \psi)(x)}
                               \Big) \, dx \\
& = & - \sum_{k < l} \int_\Omega \Big( 
        (\partial_k c_{kl}) \, \partial_l \varphi 
          + c_{kl} \, \partial_k \partial_l \varphi 
        - (\partial_l c_{kl}) \, \partial_k \varphi
          - c_{kl} \, \partial_l \partial_k \varphi \, 
                               \Big) \, \overline \psi  \\
& = & \int_\Omega
      \Big( - (\partial_1 c_{12}) \, \partial_2 \varphi
            - (\partial_1 c_{13}) \, \partial_3 \varphi
            - (\partial_2 c_{23}) \, \partial_3 \varphi   \\*
& & \hspace*{20mm} {}
            + (\partial_2 c_{12}) \, \partial_1 \varphi
            + (\partial_3 c_{13}) \, \partial_1 \varphi
            - (\partial_3 c_{23}) \, \partial_2 \varphi
      \Big) \, \overline \psi   \\
& = & 0,
\end{eqnarray*}
where we use Lemma~\ref{lposvec203} for the penultimate and (\ref{exposvec200.95;1})
for the last equality.
Then by density $\gota(\varphi,\psi) = 0$ for all $\varphi,\psi \in H^1(\Omega)$.
\end{exam}

\section{Weak assumption of measurability} \label{Sposvec4}

Some steps in the proof of the implication `\ref{tposvec101-1}$\Rightarrow$\ref{tposvec101-2}'
of Theorem~\ref{tposvec101} can be generalised without much additional effort.
Hence in this section we adopt the following slightly 
more general assumptions and notation.

Throughout this section let $\Omega \subset \Ri^d$, with $d \geq 2$, be open and 
$(Y,\Sigma,\nu)$ be a $\sigma$-finite measure space such 
that $L_2(Y)$ is separable.
For all $k,l \in \{ 1,\ldots,d \} $ let $C_{kl} \colon \Omega \to \cl(L_2(Y))$
be a function such that for all $f,g \in L_2(Y)$
the map $x \mapsto (C_{kl}(x) f,g)_{L_2(Y)}$ is 
a bounded measurable function from $\Omega$ into $\Ci$.
Suppose there is a $\mu > 0$ such that 
\begin{equation}
\RRe \sum_{k,l=1}^d (C_{kl}(x) f_l, f_k)_{L_2(Y)}
\geq \mu \sum_{k=1}^d \|f_k\|_{L_2(Y)}^2
\label{eSposvec4;1}
\end{equation}
for all $f \in L_2(Y)^d$ and a.e.\ $x \in \Omega$.
Let $V$ be a closed subspace of $H^1(\Omega,L_2(Y))$ such that 
$C_c^\infty(\Omega,L_2(Y)) \subset V$.
Define $\gota \colon V \times V \to \Ci$ by
\[
\gota(u,v)
= \sum_{k,l=1}^d \int_\Omega 
    (C_{kl}(x) \, (\partial_l u)(x), (\partial_k v)(x))_{L_2(Y)} \, dx
.  \]
Then $\gota$ is a continuous and elliptic sesquilinear form.
Let $A$ be the operator in $L_2(\Omega,L_2(Y))$ 
associated with $\gota$ and let $S$ be the semigroup on $L_2(\Omega,L_2(Y))$ 
generated by $-A$, see Proposition~1.51 in \cite{Ouh5}.

Note that we currently do not require that $S$ is positive.
Again our main interest lies in the cases $V = H^1_0(\Omega,L_2(Y))$
and $V = H^1(\Omega,L_2(Y))$.

We need some more notation.
In Example~\ref{xposvec103} there is a positive semigroup for which 
the operators $C_{kl}$ in the sesquilinear form $\gota$
do not leave the space $L_2(Y,\Ri)$ invariant. 
We aim to replace the operators $C_{kl}$ with new operators such that 
the new ones leave the space $L_2(Y,\Ri)$ invariant.
Let $Q \colon L_2(Y) \to L_2(Y)$ be a bounded (linear) operator.
Define $\widecheck{Q} \colon L_2(Y) \to L_2(Y)$ by
\[
\widecheck{Q} f
= \RRe (Q \RRe f) + i \RRe (Q \IIm f)
.  \]
It is straightforward to show that $\widecheck{Q}$ is again a linear map
which is bounded. 
Moreover, $\widecheck{Q}$ leaves $L_2(Y,\Ri)$ invariant.

If the semigroup $S$ leaves $L_2(Y,\Ri)$ invariant, then 
one can replace the coefficients $C_{kl}$ by $x \mapsto (C_{kl}(x)) \, \widecheck{\;}$.

\begin{prop} \label{pposvec200.3}
Suppose that the semigroup $S$ leaves $L_2(\Omega,L_2(Y,\Ri))$ invariant.
Then 
\[
\gota(u,v)
= \sum_{k,l=1}^d \int_\Omega 
    ((C_{kl}(x)) \, \widecheck{\;} \, (\partial_l u)(x), (\partial_k v)(x))_{L_2(Y)} \, dx
\]
for all $u,v \in V$.
\end{prop}
\begin{proof}
Write $V_\Ri = V \cap L_2(\Omega,L_2(Y,\Ri))$.
Then by \cite{Ouh5} Proposition~2.5 one deduces that $\RRe u \in V_\Ri$
for all $u \in V$ and $\gota(u,v) \in \Ri$ for all $u,v \in V_\Ri$.
Define $\gotb \colon V \times V \to \Ci$ by 
\[
\gotb(u,v)
= \sum_{k,l=1}^d \int_\Omega 
    ((C_{kl}(x)) \, \widecheck{\;} \, (\partial_l u)(x), (\partial_k v)(x))_{L_2(Y)} \, dx
.  \]
If $u,v \in V_\Ri$, then 
\begin{eqnarray*}
\gota(u,v)
& = & \RRe \gota(u,v)
= \sum_{k,l=1}^d \int_\Omega 
    (\RRe (C_{kl}(x) \, (\partial_l u)(x)), (\partial_k v)(x))_{L_2(Y)} \, dx  \\[0pt]
& = & \sum_{k,l=1}^d \int_\Omega 
    ((C_{kl}(x)) \, \widecheck{\;} \, (\partial_l u)(x), (\partial_k v)(x))_{L_2(Y)} \, dx
= \gotb(u,v)
.
\end{eqnarray*}
So $\gota(u,v) = \gotb(u,v)$ for all $u,v \in V_\Ri$.
Then by linearity the proposition follows.
\end{proof}

\begin{cor} \label{cposvec200.61}
Suppose the semigroup $S$ leaves $L_2(\Omega,L_2(Y,\Ri))$ invariant.
Then 
\[
C_{kl}(x) + C_{lk}(x) = (C_{kl}(x) + C_{lk}(x)) \, \widecheck{\;}
\]
for all $k,l \in \{ 1,\ldots,d \} $ and almost all $x \in \Omega$.
\end{cor}
\begin{proof}
This follows immediately from Propositions~\ref{pposvec200.3} 
and~\ref{pposvec200.8}.
\end{proof}

The operator coefficients $\Big( (C_{kl}(x)) \, \widecheck{\;} \Big)_{k,l \in \{ 1,\ldots,d \} }$
still satisfy the ellipticity condition~(\ref{eSposvec4;1}).

\begin{lemma} \label{lposvec200.62}
Let $f \in (L_2(Y))^d$.
Then 
\[
\RRe \sum_{k,l=1}^d ((C_{kl}(x)) \, \widecheck{\;} \, f_l, f_k)_{L_2(Y)}
\geq \mu \sum_{k=1}^d \|f_k\|_{L_2(Y)}^2
\]
for a.e.\ $x \in \Omega$.
\end{lemma} 
\begin{proof}
Let $Q \in \cl(L_2(Y))$.
Let $f,g \in L_2(Y)$.
Then 
\begin{eqnarray*}
\RRe (\widecheck{Q}  f, g)_{L_2(Y)}
& = & (\RRe Q \RRe f, \RRe g)_{L_2(Y)} + (\RRe Q \IIm f, \IIm g)_{L_2(Y)}  \\
& = & \RRe (Q \RRe f, \RRe g)_{L_2(Y)} + \RRe (Q \IIm f, \IIm g)_{L_2(Y)}
.
\end{eqnarray*}
Hence if $f \in (L_2(Y))^d$, then 
\begin{eqnarray*}
\RRe \sum_{k,l=1}^d ((C_{kl}(x)) \, \widecheck{\;} \, f_l, f_k)_{L_2(Y)}
& = & \RRe \sum_{k,l=1}^d \Big( (C_{kl}(x) \RRe f, \RRe g)_{L_2(Y)} + (C_{kl}(x) \IIm f, \IIm g)_{L_2(Y)} \Big) \\
& \geq & \mu \sum_{k=1}^d \|\RRe f_k\|_{L_2(Y)}^2 + \mu \sum_{k=1}^d \|\IIm f_k\|_{L_2(Y)}^2
= \mu \sum_{k=1}^d \|f_k\|_{L_2(Y)}^2
\end{eqnarray*}
for almost every $x \in \Omega$.
\end{proof}

If the semigroup $S$ is positive, then $S$ leaves $L_2(\Omega,L_2(Y,\Ri))$
invariant and Corollary~\ref{cposvec200.61} is applicable.

We recall that $C_{12}(x)$ is in general not a multiplication operator for 
a.e.\ $x \in \Omega$ and it does not leave the space $L_2(Y,\Ri)$ invariant,
even if the semigroup $S$ is positive,
see Example~\ref{xposvec103}.
Also if the semigroup $S$ is positive and if 
$C_{12}(x)$ leaves the space $L_2(Y,\Ri)$ invariant for all $x \in \Omega$,
then  in general  the operator
$C_{12}(x)$ still is {\em not} a multiplication operator for a.e.\ $x \in \Omega$.
Again see Example~\ref{xposvec103} for a counterexample 
with $V = H^1_0(\Omega,L_2(Y))$.
Nevertheless, one has the following property for the symmetric part.

\begin{prop} \label{pposvec220}
Suppose $S$ is a positive semigroup.
Then for almost all $x \in \Omega$ 
and $k,l \in \{ 1,\ldots,d \} $ the operator $C_{kl}(x) + C_{lk}(x)$ is 
a multiplication operator in $L_2(Y)$ with a real-valued function.
\end{prop}
\begin{proof}
Write $V_\Ri = V \cap L_2(\Omega,L_2(Y,\Ri))$.
Since the semigroup $S$ is positive, it follows from 
\cite{Ouh5} Theorem 2.6 1)$\Rightarrow$4) that $\RRe u \in V_\Ri$ for 
all $u \in V$ and
$\gota(u^+,u^-) \leq 0$
for all $u \in V_\Ri$.
Let $N \subset \Omega$ be the null set as in Lemma~\ref{lposvec206}.
By Corollary~\ref{cposvec200.61} we may assume that 
$C_{kl}(x) + C_{lk}(x) = (C_{kl}(x) + C_{lk}(x)) \, \widecheck{\;}$
for all $k,l \in \{ 1,\ldots,d \} $ and $x \in \Omega \setminus N$.
Let $x_0 \in \Omega \setminus N$.
Let $\tilde k, \tilde l \in \{ 1,\ldots,d \} $.
Write $Q = C_{\tilde k \tilde l}(x_0) + C_{\tilde l \tilde k}(x_0)$.
Suppose that $Q$ is not a multiplication operator in $L_2(Y)$.
We shall show that there exists a $u \in V_\Ri$ such that $\gota(u^+,u^-) > 0$,
which is a contradiction.

By Corollary~\ref{cposvec202} there exist $f \in L_2(Y,\Ri)$ and a 
measurable set $B \subset Y$
such that $f \geq 0$, $B \subset [f = 0]$, $0 < \nu(B) < \infty$
and $(Qf, \one_B)_{L_2(Y)} \neq 0$.
Since $Q = \widecheck{Q}$ it follows that 
$(Qf, \one_B)_{L_2(Y)} \in \Ri \setminus \{ 0 \} $.
Let $\varphi,\psi \in W^{1,\infty}_0((-1,1)^d)$ be as in Lemma~\ref{lposvec204}
with the choice $\tau = (Qf, \one_B)_{L_2(Y)}$.
Observe that $(-1,1)^d \subset B(0,d)$.
There exists a $\delta_0 > 0$ such that $B(x_0, d \delta_0) \subset \Omega$.
By Lemmas~\ref{lposvec206} and \ref{lposvec205}
there exists a $\delta \in (0,\delta_0]$ such that 
\begin{eqnarray*}
\lefteqn{
  \Big| \sum_{k,l=1}^d \int_{B(0,d)} ((\partial_l \varphi) \, \partial_k \psi)(z) \, 
     \Big( (C_{kl}(x_0 + \delta z) f, \one_B)_{L_2(Y)} - (C_{kl}(x_0) f, \one_B)_{L_2(Y)} \Big) \, dz  
  \Big|
} \hspace*{120mm}  \\*
& < & \frac{1}{2} \, |(Q f, \one_B)_{L_2(Y)}|^2
.  
\end{eqnarray*}
Define $\varphi_\delta,\psi_\delta \in W^{1,\infty}_c(\Ri^d)$
by $\varphi_\delta(x) = \varphi(\delta^{-1} \, (x-x_0))$ and
$\psi_\delta(x) = \psi(\delta^{-1} \, (x-x_0))$.
Define $u \in H^1_0(\Omega,L_2(Y))$ by 
$u(x) = \varphi_\delta(x) \, f - \psi_\delta(x) \, \one_B$.
Then $u^+(x) = \varphi_\delta(x) \, f$ and $u^-(x) = \psi_\delta(x) \, \one_B$
for almost every $x \in \Omega$.
Moreover, $u^+,u^- \in H^1_0(\Omega,L_2(Y)) \subset V$.
One calculates
\begin{eqnarray*}
\gota(u^+,u^-)
& = & \sum_{k,l=1}^d \int_\Omega (C_{kl}(x) \, (\partial_l u^+)(x), (\partial_k u^-)(x))_{L_2(Y)} \, dx  \\
& = & \sum_{k,l=1}^d \int_\Omega (\partial_l \varphi_\delta)(x) \, (\partial_k \psi_\delta)(x) \, 
     (C_{kl}(x) f, \one_B)_{L_2(Y)} \, dx  \\
& = & \delta^{d - 2} \sum_{k,l=1}^d \int_{B(0,d)} ((\partial_l \varphi) \, \partial_k \psi)(z) \, 
     (C_{kl}(x_0 + \delta z) f, \one_B)_{L_2(Y)} \, dz  \\
& = & \delta^{d - 2} \bigg( 
\sum_{k,l=1}^d \int_{\Ri^d} ((\partial_l \varphi) \, \partial_k \psi)(z) \, 
     (C_{kl}(x_0) f, \one_B)_{L_2(Y)} \, dz  
\\*
& &  {}
+ \sum_{k,l=1}^d \int_{B(0,d)} ((\partial_l \varphi) \, \partial_k \psi)(z) \cdot  \\*
& & \hspace*{30mm} {} \cdot
     \Big( (C_{kl}(x_0 + \delta z) f, \one_B)_{L_2(Y)} - (C_{kl}(x_0) f, \one_B)_{L_2(Y)} \Big) \, dz  
                    \bigg)
.
\end{eqnarray*}
We distinguish two cases.

\noindent
{\bf Case 1.} Suppose that $\tilde k \neq \tilde l$. 
Then 
\begin{eqnarray*}
\lefteqn{
\sum_{k,l=1}^d \int_{\Ri^d} ((\partial_l \varphi) \, \partial_k \psi)(z) \, 
     (C_{kl}(x_0) f, \one_B)_{L_2(Y)} \, dz  
} \hspace*{5mm} \\*
& = & \tau \, (C_{\tilde k \tilde l}(x_0) f, \one_B)_{L_2(Y)}
   + \tau \, (C_{\tilde l \tilde k}(x_0) f, \one_B)_{L_2(Y)} 
= \tau \, (Q f, \one_B)_{L_2(Y)}
= |(Q f, \one_B)_{L_2(Y)}|^2
.
\end{eqnarray*}

\noindent
{\bf Case 2.} Suppose that $\tilde k = \tilde l$. 
Then similarly 
\[
\sum_{k,l=1}^d \int_{\Ri^d} ((\partial_l \varphi) \, \partial_k \psi)(z) \, 
     (C_{kl}(x_0) f, \one_B)_{L_2(Y)} \, dz  
= \tau \, (C_{\tilde k \tilde k}(x_0) f, \one_B)_{L_2(Y)}
= \frac{1}{2} \, |(Q f, \one_B)_{L_2(Y)}|^2
.  \]

So in both cases 
\begin{eqnarray*}
\lefteqn{
\sum_{k,l=1}^d \int_{\Ri^d} ((\partial_l \varphi) \, \partial_k \psi_\delta)(z) \, 
     (C_{kl}(x_0) f, \one_B)_{L_2(Y)} \, dz  
} \hspace*{10mm} \\*
& \geq & \frac{1}{2} \, |(Q f, \one_B)_{L_2(Y)}|^2  \\
& > & \Big| \sum_{k,l=1}^d \int_{B(0,d)} ((\partial_l \varphi) \, \partial_k \psi)(z) \, 
     \Big( (C_{kl}(x_0 + \delta z) f, \one_B)_{L_2(Y)} - (C_{kl}(x_0) f, \one_B)_{L_2(Y)} \Big) \, dz 
  \Big|
\end{eqnarray*}
by the choice of $\delta$.
Hence $\gota(u^+,u^-) > 0$.
This gives the required contradiction.

We proved that $C_{\tilde k \tilde l}(x_0) + C_{\tilde l \tilde k}(x_0)$
is a multiplication operator in $L_2(Y)$ for all $x_0 \in \Omega \setminus N$.
Since $C_{kl}(x) + C_{lk}(x) = (C_{kl}(x) + C_{lk}(x)) \, \widecheck{\;}$
maps real-valued functions into real-valued functions, it follows that 
the multiplication operator is the multiplication with a real-valued function.
\end{proof}

We finish this section with an extension of the implication
`\ref{tposvec101-2}$\Rightarrow$\ref{tposvec101-1}' in Theorem~\ref{tposvec101}.

\begin{prop} \label{pposvec420}
Let $\Omega \subset \Ri^d$ be open and 
$(Y,\Sigma,\nu)$ be a $\sigma$-finite measure space such 
that $L_2(Y)$ is separable.
For all $k,l \in \{ 1,\ldots,d \} $ let
$c_{kl} \colon \Omega \times Y \to \Ri$ be a bounded measurable function.
Let $V$ be a closed subspace of $H^1(\Omega,L_2(Y))$ such that 
$C_c^\infty(\Omega,L_2(Y)) \subset V$.
Suppose that $(\RRe u)^+ \in V$ for all $u \in V$.
Let $\gotb \colon V \times V \to \Ci$ be a continuous elliptic accretive
sesquilinear form.
Let $B$ be the operator in $L_2(\Omega,L_2(Y))$
associated with the form $\gotb$ and $T$ the semigroup
generated by $-B$.
Suppose that 
\[
\gotb(u,v)
= \sum_{k,l=1}^d \int_\Omega \int_Y
      c_{kl}(x,y) \, ((\partial_l u)(x))(y) \, \overline{((\partial_k v)(x))(y)} \, dy \, dx
\]
for all $u,v \in V$.
Then $T$ is a positive semigroup.
\end{prop}
\begin{proof}
Let $u \in V$.
Then $(\RRe u)^+ \in V$ by assumption.
Moreover, 
\[
\RRe \gotb( (\RRe u)^+ , u - (\RRe u)^+ )
= - \gotb( (\RRe u)^+ , (\RRe u)^- )  
= 0
.  \]
since $\partial_l ((\RRe u)^+) = \one_{[\RRe u > 0]} \, \partial_l \RRe u$ 
and $\partial_k ((\RRe u)^-) = \one_{[\RRe u < 0]} \, \partial_k \RRe u$.
Then the semigroup $T$ leaves the closed convex set
$L_2(Y)^+$ invariant by \cite{Ouh5} Theorem 2.2 2)$\Rightarrow$1).
So $T$ is a positive semigroup.
\end{proof}

\section{Positive semigroups} \label{Sposvec5}

In this section we prove the implication `\ref{tposvec101-1}$\Rightarrow$\ref{tposvec101-2}'
in Theorem~\ref{tposvec101} and various extensions.
Throughout this section we make the following assumptions.

Let $\Omega \subset \Ri^d$ be open and 
$(Y,\Sigma,\nu)$ be a $\sigma$-finite measure space such 
that $L_2(Y)$ is separable.
For all $k,l \in \{ 1,\ldots,d \} $ let $C_{kl} \colon \Omega \to \cl(L_2(Y))$
be a function such that $x \mapsto (C_{kl}(x) f,g)_{L_2(Y)}$ is 
a bounded measurable function from $\Omega$ into $\Ci$ 
for all $f,g \in L_2(Y)$.
Suppose there is a $\mu > 0$ such that 
\[
\RRe \sum_{k,l=1}^d (C_{kl}(x) f_l, f_k)_{L_2(Y)}
\geq \mu \sum_{k=1}^d \|f_k\|_{L_2(Y)}^2
\]
for all $f \in L_2(Y)^d$ and a.e.\ $x \in \Omega$.
Let $V_\Omega$ be a closed subspace of $H^1(\Omega)$ such that 
$C_c^\infty(\Omega) \subset V_\Omega$.
Suppose $(\RRe \varphi)^+ \in V_\Omega$ for all $\varphi \in V_\Omega$.
Let $V$ be the closure in $H^1(\Omega,L_2(Y))$ of 
\[
\spann \{ \varphi \otimes f : 
    \varphi \in V_\Omega \mbox{ and } f \in L_2(Y) \} 
.  \]
Note that if $V_\Omega = H^1_0(\Omega)$, then $V = H^1_0(\Omega, L_2(Y))$,
and if $V_\Omega = H^1(\Omega)$, then $V = H^1(\Omega, L_2(Y))$.
Define $\gota \colon V \times V \to \Ci$ by
\[
\gota(u,v)
= \sum_{k,l=1}^d \int_\Omega 
    (C_{kl}(x) \, (\partial_l u)(x), (\partial_k v)(x))_{L_2(Y)} \, dx
.  \]
Let $A$ be the operator in $L_2(\Omega,L_2(Y))$
associated with $\gota$ and $S$ the semigroup
generated by $-A$.

\begin{thm} \label{tposvec405}
Suppose the semigroup $S$ is positive.
Moreover, suppose at least one of the following conditions.
\begin{tabelR}
\item \label{tposvec405-1}
The form $\gota$ is symmetric.
\item \label{tposvec405-2}
For all $k,l \in \{ 1,\ldots,d \} $ the function $C_{kl} \colon \Omega \to \cl(L_2(Y))$
is separably valued.
\item \label{tposvec405-3}
$d = 2$ and $V_\Omega = H^1_0(\Omega)$.
\item \label{tposvec405-4}
$d = 2$, $V_\Omega = H^1(\Omega)$ and $\Omega$ is bounded.
\item \label{tposvec405-5}
$L_2(Y) = \ell_2$.
\end{tabelR}
Then for all $k,l \in \{ 1,\ldots,d \} $ there exist bounded measurable functions
$c_{kl} \colon \Omega \times Y \to \Ri$ such that
\[
\gota(u,v)
= \sum_{k,l=1}^d \int_\Omega \int_Y
      c_{kl}(x,y) \, ((\partial_l u)(x))(y) \, \overline{((\partial_k v)(x))(y)} \, dy \, dx
\]
for all $u,v \in V$.
\end{thm}

After a preparing lemma we will give the proof of Theorem~\ref{tposvec405}
under the assumptions~\ref{tposvec405-1}, \ref{tposvec405-2}, \ref{tposvec405-3}
or \ref{tposvec405-4}.
Its proof under assumption~\ref{tposvec405-5} is postponed to Section~\ref{Sposvec6}.

Define $V_\Omega(\Ri) = V_\Omega \cap H^1(\Omega,\Ri)$ and 
$V_\Omega^+ = V_\Omega \cap H^1(\Omega)^+$.
Then $V_\Omega^+ \subset V_\Omega(\Ri) \subset V_\Omega$ by assumption and 
linearity.
A major step in the proof of Theorem~\ref{tposvec405} is the following 
equality.

\begin{lemma} \label{lposvec401}
Suppose the semigroup $S$ is positive.
Let $\varphi,\psi \in V_\Omega$ and $f,g \in L_2(Y)$ with $f \, g = 0$.
Then $\gota(\varphi \otimes f, \psi \otimes g) = 0$.
\end{lemma}
\begin{proof}
The operators $C_{kl}(x) + C_{lk}(x)$ are multiplication operators on $L_2(Y)$
for almost every $x \in \Omega$ and all $k,l \in \{ 1,\ldots,d \} $
by Proposition~\ref{pposvec220}.
Let $\varphi,\psi \in V_\Omega$ and $f,g \in L_2(Y)$ with $f \, g = 0$.
Then
\begin{eqnarray}
\gota(\varphi \otimes f, \psi \otimes g)
& = & \sum_{k,l=1}^d \int_\Omega (C_{kl}(x) f, g)_{L_2(Y)} \, 
      (\partial_l \varphi)(x) \, \overline{(\partial_k \psi)(x)} \, dx  \label{elposvec401;1}  \\
& = & \sum_{k<l} \int_\Omega (C_{kl}(x) f, g)_{L_2(Y)} \, 
      (\partial_l \varphi)(x) \, \overline{(\partial_k \psi)(x)} \, dx  \nonumber  \\*
& & \hspace*{10mm} {}
   + \sum_{k<l} \int_\Omega (C_{lk}(x) f, g)_{L_2(Y)} \, 
      (\partial_k \varphi)(x) \, \overline{(\partial_l \psi)(x)} \, dx   \nonumber \\
& = & \sum_{k<l} \int_\Omega (C_{kl}(x) f, g)_{L_2(Y)} \, 
      \Big( (\partial_l \varphi)(x) \, \overline{(\partial_k \psi)(x)} 
          - (\partial_k \varphi)(x) \, \overline{(\partial_l \psi)(x)} \Big) \, dx  \nonumber 
,
\end{eqnarray}
where we used that 
\[
(C_{lk}(x) f, g)_{L_2(Y)} + (C_{kl}(x) f, g)_{L_2(Y)}
= ((C_{kl}(x) + C_{lk}(x)) f, g)_{L_2(Y)}
= 0
.  \]
Now let $\varphi,\psi \in V_\Omega^+$ and $f,g \in L_2(Y)^+$ with $f \, g = 0$.
Define $u = \varphi \otimes f - \psi \otimes g$.
Then $u \in V$, $u^+ = \varphi \otimes f$ and $u^- = \psi \otimes g$.
Since the semigroup $S$ is positive, it follows from 
\cite{Ouh5} Theorem 2.6 1)$\Rightarrow$4) that
$\gota(u^+,u^-) \leq 0$.
So 
\begin{eqnarray*}
0
& \geq & \gota(u^+,u^-)
= \gota(\varphi \otimes f, \psi \otimes g)  \\
& = & \sum_{k<l} \int_\Omega (C_{kl}(x) f, g)_{L_2(Y)} \, 
      \Big( (\partial_l \varphi)(x) \, (\partial_k \psi)(x) 
          - (\partial_k \varphi)(x) \, (\partial_l \psi)(x) \Big) \, dx 
\end{eqnarray*}
by (\ref{elposvec401;1}).
Swapping $\varphi$ and $\psi$ gives
\[
0 
= \sum_{k<l} \int_\Omega (C_{kl}(x) f, g)_{L_2(Y)} \, 
      \Big( (\partial_l \varphi)(x) \, (\partial_k \psi)(x) 
          - (\partial_k \varphi)(x) \, (\partial_l \psi)(x) \Big) \, dx 
= \gota(\varphi \otimes f, \psi \otimes g)
,  \]
where we again used (\ref{elposvec401;1}) in the last step.
Then by linearity the lemma follows.
\end{proof}

\begin{proof}[{\bf Proof of Theorem~\ref{tposvec405}.}]
By Proposition~\ref{pposvec200.3} we may assume that $C_{kl}(x) = (C_{kl}(x)) \, \widecheck{\;}$
for all $k,l \in \{ 1,\ldots,d \} $ and almost all $x \in \Omega$.

`\ref{tposvec405-1}'.
We now assume in addition that the form $\gota$ is symmetric.
Let $k,l \in \{ 1,\ldots,d \} $.
By Proposition~\ref{pposvec220} the operator $C_{kl}(x) + C_{lk}(x)$ 
is a multiplication operator on $L_2(Y)$ 
with a real-valued function for almost every $x \in \Omega$.
Hence by Proposition~\ref{pposvec230}\ref{pposvec230-1}$\Rightarrow$\ref{pposvec230-2}
there exist bounded measurable functions $c_{kl} = c_{lk} \colon \Omega \times Y \to \Ri$
such that for all $f \in L_2(Y)$ one has 
\[
\Big( (C_{kl}(x) + C_{lk}(x)) f \Big) (y)
= (c_{kl}(x,y) + c_{lk}(x,y)) \, f(y)
\quad \mbox{for a.e.\ } (x,y) \in \Omega \times Y
.  \]
Define $\gotb \colon V \times V \to \Ci$ by 
\[
\gotb(u,v)
= \sum_{k,l=1}^d \int_\Omega \int_Y
      c_{kl}(x,y) \, ((\partial_l u)(x))(y) \, \overline{((\partial_k v)(x))(y)} \, dy \, dx
.  \]
Let $f \in L_2(Y)$ and $\varphi \in V_\Omega(\Ri)$.
Then 
\begin{eqnarray*}
\gota(\varphi \otimes f, \varphi \otimes f)
& = & \sum_{k,l=1}^d \int_\Omega (C_{kl}(x) f, f)_{L_2(Y)} \,
      (\partial_l \varphi)(x) \, (\partial_k \varphi)(x) \, dx  \\
& = & \frac{1}{2} \sum_{k,l=1}^d \int_\Omega ((C_{kl}(x) + C_{lk}(x)) f, f)_{L_2(Y)} \,
      (\partial_l \varphi)(x) \, (\partial_k \varphi)(x) \, dx  \\
& = & \sum_{k,l=1}^d \int_\Omega \int_Y c_{kl}(x,y) \, |f(y)|^2  \,
      (\partial_l \varphi)(x) \, (\partial_k \varphi)(x) \, dx  \\
& = & \gotb(\varphi \otimes f, \varphi \otimes f)
.
\end{eqnarray*}
Next let $f \in L_2(Y)$ and $\varphi,\psi \in V_\Omega(\Ri)$.
Since both the forms $\gota$ and $\gotb$ are symmetric, one deduces by polarisation
\begin{eqnarray*}
\gota(\varphi \otimes f, \psi \otimes f)
& = & \frac{1}{4} \Big( \gota((\varphi+\psi) \otimes f, (\varphi+\psi) \otimes f)
   - \gota((\varphi-\psi) \otimes f, (\varphi-\psi) \otimes f)
              \Big)  \\
& = & \frac{1}{4} \Big( \gotb((\varphi+\psi) \otimes f, (\varphi+\psi) \otimes f)
   - \gotb((\varphi-\psi) \otimes f, (\varphi-\psi) \otimes f)
              \Big)  \\
& = & \gotb(\varphi \otimes f, \psi \otimes f)
.  
\end{eqnarray*}
Now let $\varphi,\psi \in V_\Omega(\Ri)$ and 
$A,B \subset Y$ be measurable with $\nu(A) < \infty$ and $\nu(B) < \infty$.
If $A \cap B = \emptyset$, then
\[
\gota(\varphi \otimes \one_A, \psi \otimes \one_B)
= 0 
= \gotb(\varphi \otimes \one_A, \psi \otimes \one_B)
\]
by Lemma~\ref{lposvec401}.
Therefore without the condition $A \cap B = \emptyset$ one obtains
\begin{eqnarray*}
\gota(\varphi \otimes \one_A, \psi \otimes \one_B)
& = & \gota(\varphi \otimes \one_{A \cap B}, \psi \otimes \one_{A \cap B})  \\
& = & \gotb(\varphi \otimes \one_{A \cap B}, \psi \otimes \one_{A \cap B})  
= \gotb(\varphi \otimes \one_A, \psi \otimes \one_B)
.
\end{eqnarray*}
Hence by linearity and density one deduces that 
$\gota(u,v) = \gotb(u,v)$ for all $u,v \in V$
and the theorem is established under Condition~\ref{tposvec405-1}.

\smallskip

For the proof of Theorem~\ref{tposvec405} under the Condition~\ref{tposvec405-2}
we need some preparation.

Let $\cm$ be the subspace of $\cl(L_2(Y))$ of all bounded multiplication operators
and let $\cm_\Ri$ be the subset of multiplication operators
corresponding to multiplication by a bounded real-valued function.
Further let $\cl^r(L_2(Y))$ be the linear span of all positive operators 
on $L_2(Y)$.
Note that $\cm \subset \cl^r(L_2(Y))$.
It follows from Theorem~\ref{tposvec201}\ref{tposvec201-1}$\Leftrightarrow$\ref{tposvec201-6}
and \cite{Zaa} Theorem~12.2(ii) that there exists a unique bounded linear
projection $P_r \colon \cl^r(L_2(Y)) \to \cm$ onto $\cm$
such that $P_r(T) \in \cm_\Ri$ for all $T \in \cl^r(L_2(Y))$ 
which leave $L_2(Y,\Ri)$ invariant and such that $0 \leq P_r(T) \leq T$
for all $0 \leq T \in \cl^r(L_2(Y))$.
By \cite{Voi2} Theorem~1.4 one has $\|P_r(T)\| \leq \|T\|$ for all $T \in \cl^r(L_2(Y))$.

Since $\cm$ is isomorphic to $L_\infty(Y)$, it is an injective Banach space.
Hence the map $P_r$ can be 
extended to a contractive linear operator $P \colon \cl(L_2(Y)) \to \cm$,
see \cite{Voi2} Remark~1.5(a).
Moreover, the extension can be made such that $P(T) \in \cm_\Ri$ for all $T \in \cl(L_2(Y))$ 
which leave $L_2(Y,\Ri)$ invariant.
Since $P(T) \in \cm$, one has $P(P(T)) = P(T)$ for all $T \in \cl(L_2(Y))$ and 
$P$ is a projection.

`\ref{tposvec405-2}'.
By the Hahn--Banach theorem a subspace of a dual space $Z^*$ is weakly$^*$ dense in $Z^*$
if and only if it separates the points of $Z$. 
We apply this with $Z = \cl(L_2(Y))$.
Since $x \mapsto (C_{kl}(x) f,g)_{L_2(Y)}$ is 
a measurable function from $\Omega$ into $\Ci$ 
for all $f,g \in L_2(Y)$, the separability condition of \ref{tposvec405-2}
implies that the function $C_{kl} \colon \Omega \to \cl(L_2(Y))$
is Bochner measurable for all $k,l \in \{ 1,\ldots,d \} $
by \cite{HNVW1} Theorem~1.1.6(3)$\Rightarrow$(1) (or \cite{Are7} in case $d = 1$).

For all $\varphi,\psi \in V_\Omega$ define the operator 
$T_{\varphi,\psi} \in \cl(L_2(Y))$ by 
\[
T_{\varphi,\psi}
= \sum_{k,l=1}^d \int_\Omega (\partial_l \varphi)(x) \, \overline{(\partial_k \psi)(x)} \, C_{kl}(x) \, dx
.  \]
Then
$\gota(\varphi \otimes f,\psi \otimes g) = (T_{\varphi,\psi} f, g)_{L_2(Y)}$ for all $f,g \in L_2(Y)$.
If $f,g \in L_2(Y)$ with $f \, g = 0$, then 
$(T_{\varphi,\psi} f, g)_{L_2(Y)} = \gota(\varphi \otimes f,\psi \otimes g) = 0$
by Lemma~\ref{lposvec401}.
Hence $T_{\varphi,\psi}$ is a multiplication operator.
In particular
\[
T_{\varphi,\psi}
= P(T_{\varphi,\psi})
= \sum_{k,l=1}^d \int_\Omega (\partial_l \varphi)(x) \, \overline{(\partial_k \psi)(x)} \, P(C_{kl}(x)) \, dx
.  \]
Here we use that the $C_{kl}$ are Bochner integrable.
Let $k,l \in \{ 1,\ldots,d \} $.
Then $P(C_{kl}(x)) = P((C_{kl}(x)) \, \widecheck{\;}) \in \cm_\Ri$
for almost every $x \in \Omega$.
By Proposition~\ref{pposvec230}\ref{pposvec230-1}$\Rightarrow$\ref{pposvec230-2}
there exists a bounded measurable function $c_{kl} \colon \Omega \times Y \to \Ri$ 
such that for all $f \in L_2(Y)$ one has 
\[
\Big( P(C_{kl}(x)) f \Big) (y)
= c_{kl}(x,y) \, f(y)
\quad \mbox{for a.e.\ } (x,y) \in \Omega \times Y
.  \]
Define $\gotb \colon V \times V \to \Ci$ by 
\[
\gotb(u,v)
= \sum_{k,l=1}^d \int_\Omega \int_Y
      c_{kl}(x,y) \, ((\partial_l u)(x))(y) \, \overline{((\partial_k v)(x))(y)} \, dy \, dx
.  \]
If $\varphi,\psi \in V_\Omega$ and $f,g \in L_2(Y)$, then 
\begin{eqnarray*}
\gota(\varphi \otimes f,\psi \otimes g) 
& = & (T_{\varphi,\psi} f, g)_{L_2(Y)}  \\
& = & \sum_{k,l=1}^d \int_\Omega (\partial_l \varphi)(x) \, \overline{(\partial_k \psi)(x)} \,
     (P(C_{kl}(x)) f, g)_{L_2(Y)} \, dx  \\
& = & \sum_{k,l=1}^d \int_\Omega (\partial_l \varphi)(x) \, \overline{(\partial_k \psi)(x)} \,
     \int_Y c_{kl}(x,y) \, f(y) \, \overline{g(y)} \, dy \, dx  \\
& = & \gotb(\varphi \otimes f,\psi \otimes g) 
.
\end{eqnarray*}
Hence by linearity and density one deduces that 
$\gota(u,v) = \gotb(u,v)$ for all $u,v \in V$
and the theorem is established under Condition~\ref{tposvec405-2}.

`\ref{tposvec405-3}'. 
Now we prove Theorem~\ref{tposvec405} under Condition~\ref{tposvec405-3}.
Let $f,g \in L_2(Y)$ with $f \, g = 0$.
Let $\tau \in C_c^\infty(\Omega)$.
There exist $\psi,\tilde \psi \in C_c^\infty(\Omega)$ such that 
$\psi(x) = x_1$ and $\tilde \psi(x) = x_2$ for all $x \in \supp \tau$.
Then Lemma~\ref{lposvec401} and (\ref{elposvec401;1}) give
\[
0 
= \gota(\tau \otimes f, \psi \otimes g)
= \int_\Omega (C_{12}(x) f, g)_{L_2(Y)} \, (\partial_2 \tau)(x) \, dx
\]
and 
\[
0 
= \gota(\tilde \psi \otimes f, \overline \tau \otimes g)
= \int_\Omega (C_{12}(x) f, g)_{L_2(Y)} \, (\partial_1 \tau)(x) \, dx
.  \]
So $x \mapsto (C_{12}(x) f, g)_{L_2(Y)}$ is an element of 
$W^{1,\infty}(\Omega)$ with vanishing gradient.
Hence $x \mapsto (C_{12}(x) f, g)_{L_2(Y)}$ is constant a.e.\ on connected 
components of $\Omega$ by \cite{Zie2} Corollary~2.1.9.

Let $N \subset \Omega$ be the null set as in Lemma~\ref{lposvec206}.
Fix $x_0 \in \Omega \setminus N$.
If $x \in \Omega \setminus N$ and $f,g \in L_2(Y)$ with $f \, g = 0$, then 
$((C_{12}(x) - C_{12}(x_0))f, g)_{L_2(Y)} = 0$, so 
$C_{12}(x) - C_{12}(x_0)$ is a multiplication operator on $L_2(Y)$
by Theorem~\ref{tposvec201}\ref{tposvec201-2}$\Rightarrow$\ref{tposvec201-1} 
and it leaves $L_2(Y,\Ri)$ invariant.
Hence by Proposition~\ref{pposvec230}\ref{pposvec230-1}$\Rightarrow$\ref{pposvec230-2}
there exists a bounded measurable function $c_{12} \colon \Omega \times Y \to \Ri$ 
such that for all $f \in L_2(Y)$ one has 
\[
\Big( (C_{12}(x) - C_{12}(x_0)) f \Big) (y)
= c_{12}(x,y) \, f(y)
\quad \mbox{for a.e.\ } (x,y) \in \Omega \times Y
.  \]
As in the proof of Statement~\ref{tposvec405-1}
there are bounded measurable functions $c_{11}, c_{22}, c \colon \Omega \times Y \to \Ri$ 
such that for all $f \in L_2(Y)$ one has 
\[
\begin{array}{r@{}c@{}l}
\Big( C_{11}(x) f \Big) (y)
& {} = {} & c_{11}(x,y) \, f(y)  \\[5pt]
\Big( C_{22}(x) f \Big) (y)
& {} = {} & c_{22}(x,y) \, f(y)  \\[5pt]
\Big( (C_{12}(x) + C_{21}(x)) f \Big) (y)
& {} = {} & c(x,y) \, f(y)  
\end{array}
\quad \mbox{for a.e.\ } (x,y) \in \Omega \times Y
.  \]
Define $c_{21} = c - c_{12} \colon \Omega \times Y \to \Ri$
and define $\gotb \colon H^1_0(\Omega,L_2(Y)) \times H^1_0(\Omega,L_2(Y)) \to \Ci$ by 
\[
\gotb(u,v)
= \sum_{k,l=1}^d \int_\Omega \int_Y
      c_{kl}(x,y) \, ((\partial_l u)(x))(y) \, \overline{((\partial_k v)(x))(y)} \, dy \, dx
.  \]
If $\varphi,\psi \in C_c^\infty(\Omega)$ and $f,g \in L_2(Y)$, then 
\begin{eqnarray*}
\gota(\varphi \otimes f, \psi \otimes g)
& = & \gotb(\varphi \otimes f, \psi \otimes g)
   + \int_\Omega ( C_{12}(x_0) f, g)_{L_2(Y)} 
          \, (\partial_2 \varphi)(x) \, \overline{(\partial_1 \psi)(x)} \, dx   \\*
& & \hspace*{10mm} {}
   - \int_\Omega ( C_{12}(x_0) f, g)_{L_2(Y)}   
          \, (\partial_1 \varphi)(x) \, \overline{(\partial_2 \psi)(x)} \, dx  \\
& = & \gotb(\varphi \otimes f, \psi \otimes g)
\end{eqnarray*}
by Lemma~\ref{lposvec203}.
Hence by linearity and density, see Lemma~\ref{lposvec213}, one deduces that 
$\gota(u,v) = \gotb(u,v)$ for all $u,v \in H^1_0(\Omega,L_2(Y))$
and the theorem is established under Condition~\ref{tposvec405-3}.

`\ref{tposvec405-4}'.
Let $f,g \in L_2(Y)$ with $f \, g = 0$.
It follows as in the Dirichlet case in Statement~\ref{tposvec405-3}
that $x \mapsto (C_{12}(x) f, g)_{L_2(Y)}$ is constant on connected 
components of $\Omega$.
Let $\Omega_0$ be a connected component of $\Omega$.
Define $\varphi,\psi \in H^1(\Omega)$ by $\varphi(x) = x_2$ and $\psi(x) = x_1$
if $x \in \Omega_0$, and $\varphi(x) = \psi(x) = 0$ if $x \in \Omega \setminus \Omega_0$.
Then Lemma~\ref{lposvec401} and (\ref{elposvec401;1}) give
\[
0 
= \gota(\varphi \otimes f, \psi \otimes g)
= \int_{\Omega_0} (C_{12}(x) f, g)_{L_2(Y)} \, dx
\]
So $(C_{12}(x) f, g)_{L_2(Y)} = 0$ for almost every $x \in \Omega_0$ and 
then also $(C_{12}(x) f, g)_{L_2(Y)} = 0$ for a.e.\ $x \in \Omega$.
So $C_{12}(x)$ is a multiplication operator for a.e.\ $x \in \Omega$
and the rest of the proof is clear.
\end{proof}

\begin{remark} \label{rposvec405.5}
For Theorem~\ref{tposvec405} to hold it suffices that for each $\rho \in L_1(\Omega)$
the functions $x \mapsto \rho(x) \, C_{kl}(x)$ from $\Omega$ into $\cl(L_2(Y))$ 
are Pettis integrable for all $k,l \in \{ 1,\ldots,d \} $.
See \cite{HNVW1} Section~1.2.c for an introduction to the 
Pettis integral.
Then 
\[
\Phi(  \int_\Omega \rho(x) \, C_{kl}(x) \, dx )
= \int_\Omega \rho(x) \, \Phi( C_{kl}(x) ) \, dx
\]
for each $\Phi \in \cl(L_2(Y))^*$
and hence 
\[
P(  \int_\Omega \rho(x) \, C_{kl}(x) \, dx )
= \int_\Omega \rho(x) \, P( C_{kl}(x) ) \, dx
.  \]
This is what is needed in the proof of Theorem~\ref{tposvec405}.
We will see in Example~\ref{xposvec610} that under 
our general conditions $C_{kl}$ may not be Pettis integrable.
\end{remark}

\begin{remark} \label{rposvec406}
Note that due to Proposition~\ref{pposvec200.9}
we proved in case $d=2$, $V_\Omega = H^1(\Omega)$ and $\Omega$ bounded
that $C_{kl}(x)$ is a multiplication operator on $L_2(Y)$ 
for all $k,l \in \{ 1,2 \} $ and 
almost every $x \in \Omega$.
In general $\Omega$ has to be bounded, see Example~\ref{xposvec103}.
\end{remark}

The above remark is no longer valid in higher dimensions by the 
next example.

\begin{exam} \label{xposvec407}
Let $\Omega = (-1,1)^3$ and let the functions $c_{kl} \colon \Omega \to \Ri$
be as in Example~\ref{xposvec200.95}.
Let $Y = \{ 4,5 \} $ with counting measure.
For all $k,l \in \{1,2,3 \} $ define $C_{kl} \colon \Omega \to \cl(L_2(Y))$
by 
\begin{eqnarray*}
\Big( C_{kl}(x) f \Big)(4) & = & c_{kl}(x) \, f(5) \mbox{ if } k \neq l,   \\
\Big( C_{kl}(x) f \Big)(5) & = & 0  \mbox{ if } k \neq l,   \\
\Big( C_{kl}(x) f \Big)(n) & = & 6 \, f(n)  \mbox{ if } k = l \mbox{ and } n \in Y .   
\end{eqnarray*}
Then $C_{kl}$ is Bochner measurable.
Moreover, the ellipticity condition is satisfied and 
$\gota(u,v) 
= \sum_{k=1}^3 \int_\Omega ((\partial_k u)(x) \, (\partial_k v)(x))_{L_2(Y)} \, dx$
for all $u,v \in H^1(\Omega,L_2(Y))$.
So the associated semigroup is positive.
But for almost all $x \in \Omega$ the operator
$C_{12}(x)$ is not a multiplication operator on $L_2(Y)$.
\end{exam}

\begin{proof}[{\bf Proof of Theorem~\ref{tposvec101}.}]
This follows from Theorem~\ref{tposvec405}\ref{tposvec405-2}
with $V_\Omega = H^1_0(\Omega)$ or $V_\Omega = H^1(\Omega)$,
together with Lemma~\ref{lposvec213}.
\end{proof}

\section{\texorpdfstring{$\ell_2$}{l2}-valued coefficient operators} \label{Sposvec6}

We do not know whether the separability 
assumption~\ref{tposvec405-2} in Theorem~\ref{tposvec405}
can be omitted.
In this section we show that it can if $Y$ is purely atomic.
We also discuss the difficulties which occur in the general case.
The main result of this section is as follows.
Let $\ell_2$ be the usual sequence space and for all $n \in \Ni$ we denote by 
$e_n$ the usual $n$-th unit vector in $\ell_2$.
Let $\Omega \subset \Ri^d$ be open. 
We will identify $L_2(\Omega,\ell_2)$ with $\ell_2(L_2(\Omega))$ 
by associating $u \in L_2(\Omega,\ell_2)$ with $(u_n)_{n \in \Ni} \in \ell_2(L_2(\Omega))$, 
where $u_n \in L_2(\Omega)$ is given by $u_n(x) = (u(x), e_n)_{\ell_2}$ for all $n \in \Ni$.
For all $k,l \in \{ 1,\ldots,d \} $ let $C_{kl} \colon \Omega \to \cl(\ell_2)$
be a function such that $x \mapsto (C_{kl}(x) f,g)_{\ell_2}$ is 
a bounded measurable function from $\Omega$ into $\Ci$ 
for all $f,g \in \ell_2$.
Further we assume that there is a $\mu > 0$ such that 
\[
\RRe \sum_{k,l=1}^d (C_{kl}(x)  f_l, f_k)_{\ell_2}
\geq \mu \sum_{k=1}^d \|f_k\|_{\ell_2}^2
\]
for all $x \in \Omega$ and $f_1,\ldots,f_d \in \ell_2$.
As in Section~\ref{Sposvec5} let $V_\Omega$ be a closed subspace of $H^1(\Omega)$ such that 
$C_c^\infty(\Omega) \subset V_\Omega$.
Suppose $(\RRe \varphi)^+ \in V_\Omega$ for all $\varphi \in V_\Omega$.
Let $V = \ell_2(V_\Omega)$.
Define the form $\gota \colon V \times V \to \Ci$
by 
\[
\gota(u,v) 
= \sum_{k,l=1}^d \int_\Omega
    ( C_{kl}(x) \, (\partial_l u)(x) , (\partial_k v)(x) )_{\ell_2} \, dx
.  \]
Then $\gota$ is continuous and 
$\RRe \gota(u,u) + \mu \, \|u\|_{L_2(\Omega,\ell_2)}^2 \geq \mu \, \|u\|_V^2$
for all $u \in V$, so $\gota$ is elliptic.
Let $A$ be the operator in $L_2(\Omega,\ell_2)$ associated with $\gota$ and 
let $S$ be the $C_0$-semigroup on $L_2(\Omega,\ell_2)$ generated by $-A$.
Let $M \geq 0$ be such that $\|C_{kl}(x)\|_{\cl(\ell_2)} \leq M$ for all 
$x \in \Omega$ and $k,l \in \{ 1,\ldots,d \} $.

\begin{thm} \label{tposvec606}
Adopt the above notation and assumptions.
The following are equivalent.
\begin{tabeleq} 
\item \label{tposvec606-1}
The semigroup $S$ is positive.
\item \label{tposvec606-2}
For all $n \in \Ni$ there exists a semigroup $S^{(n)}$ associated to a 
form $\gota_n \colon V_\Omega \times V_\Omega \to \Ci$ given 
by 
\[
\gota_n(\varphi,\psi)
= \sum_{k,l=1}^d \int_\Omega c^{(n)}_{kl} \, (\partial_l \varphi) \, \overline{\partial_k \psi}
,  \]
where $c^{(n)}_{kl} \colon \Omega \to \Ri$ is a measurable function
with $\|c^{(n)}_{kl}\|_{L_\infty(\Omega)} \leq M$
for all $k,l \in \{ 1,\ldots,d \} $ 
and $\RRe \sum_{k,l=1}^d c^{(n)}_{kl}(x) \, \xi_l \, \overline{\xi_k} \geq \mu \, |\xi|^2$
for all $\xi \in \Ci^d$ and $x \in \Omega$.
Moreover, 
\[
S_t u = ( S^{(n)}_t u_n )_{n \in \Ni}
\]
for all $t > 0$ and $u \in \ell_2(L_2(\Omega))$.
\item \label{tposvec606-3}
For all $k,l \in \{ 1,\ldots,d \} $ and $n \in \Ni$ there exists a  bounded measurable function
$c^{(n)}_{kl} \colon \Omega \to \Ri$ such that
\[
\gota(u,v)
= \sum_{k,l=1}^d \int_\Omega \sum_{n=1}^\infty
      c^{(n)}_{kl}(x) \, (\partial_l u_n)(x) \, \overline{(\partial_k v_n)(x)} \, dx
\]
for all $u,v \in V$.
Moreover,
$|c^{(n)}_{kl}(x)| \leq M$ for all $k,l \in \{ 1,\ldots,d \} $, $x \in \Omega$ and $n \in \Ni$.
Further $\RRe \sum_{k,l=1}^d c^{(n)}_{kl}(x) \, \xi_l \, \overline{\xi_k} \geq \mu \, |\xi|^2$
for all $\xi \in \Ci^d$, $x \in \Omega$ and $n \in \Ni$.
\end{tabeleq}
\end{thm}

For the proof we need some preparation.

\smallskip

Let $Z$ be a Banach space and let $F \colon \Omega \to Z^*$ be a function.
We say that $F$ is {\bf weak$^*$-integrable} if 
$x \mapsto \langle F(x), z \rangle_{Z^* \times Z}$ is measurable for all $z \in Z$ and 
there exists a $\varphi \in L_1(\Omega)$ such that $\|F(x)\|_{Z^*} \leq \varphi(x)$
for almost every $x \in \Omega$.
If $F$ is weak$^*$-integrable, then we define $\int_\Omega F(x) \, dx \in Z^*$ by
\[
\langle \int_\Omega F(x) \, dx, z \rangle_{Z^* \times Z}
= \int_\Omega \langle F(x), z \rangle_{Z^* \times Z} \, dx
.  \]
We call $\int_\Omega F(x) \, dx$ the {\bf weak$^*$-integral of $F$}.

\begin{lemma} \label{lposvec602}
Let $Z$ be a Banach space and let $F \colon \Omega \to Z^*$ be weak$^*$-integrable.
Let $T \in \cl(Z)$.
Then $T^* \, F$ is weak$^*$-integrable and 
\[
T^* \int_\Omega F(x) \, dx
= \int_\Omega (T^* \, F)(x) \, dx
.  \]
\end{lemma}
\begin{proof}
It is easy to see that $T^* \, F$ is weak$^*$-integrable.
If $z \in Z$, then 
\begin{eqnarray*}
\langle T^* \int_\Omega F(x) \, dx, z \rangle_{Z^* \times Z}
& = & \langle \int_\Omega F(x) \, dx, T z \rangle_{Z^* \times Z}  \\
& = & \int_\Omega \langle F(x), T z \rangle_{Z^* \times Z}  \, dx
= \int_\Omega \langle (T^* \, F)(x), z \rangle_{Z^* \times Z} \, dx
\end{eqnarray*}
as required.
\end{proof}

Let $\cb_1(\ell_2)$ denote the Banach space of all trace class operators in $\ell_2$.
Define the bilinear map $\langle \cdot , \cdot \rangle_{\cb_1(\ell_2) \times \cl(\ell_2)} 
     \colon \cb_1(\ell_2) \times \cl(\ell_2) \to \Ci$ by 
\[
\langle S, T \rangle_{\cb_1(\ell_2) \times \cl(\ell_2)}
= \Tr(S \, T)
.  \]
Then $\langle \cdot , \cdot \rangle_{\cb_1(\ell_2) \times \cl(\ell_2)}$ is a duality
and 
\[
(\cb_1(\ell_2))^* = \cl(\ell_2)
,  \]
see \cite{Ped1} Theorem~3.4.13.
Define $P \colon \cl(\ell_2) \to \cl(\ell_2)$ by 
\[
P(T) f 
= \sum_{n=1}^\infty (T e_n, e_n)_{\ell_2} \, (f,e_n)_{\ell_2} \, e_n
.  \]
Then $P(T)$ is a multiplication operator for all $T \in \cl(\ell_2)$
and if $T$ is a multiplication operator, then $P(T) = T$.
Hence $P^2 = P$, that is $P$ is a projection.
We remark that $P$ is also a contraction.

\begin{lemma} \label{lposvec605}
The projection $P$ has a preadjoint.
\end{lemma}
\begin{proof}
If $T \in \cl(\ell_2)$, then observe that 
$(P(T) f, f)_{\ell_2} = \sum_{n=1}^\infty (T e_n, e_n)_{\ell_2} \, |(f, e_n)_{\ell_2}|^2$
for all $f \in \ell_2$.
Hence if $T$ is a positive operator, then again $P(T)$ is a positive operator in $\cl(\ell_2)$.
If $T \in \cl(\ell_2)$ and $n \in \Ni$, then $P(T) e_n = (T e_n, e_n)_{\ell_2} \, e_n$
and $( P(T) e_n, e_n)_{\ell_2} = (T e_n, e_n)_{\ell_2}$.
Further, if $T \in \cb_1(\ell_2)$ is a positive operator, in the Hilbert space sense, then 
$\sum_{n=1}^\infty ( P(T) e_n, e_n)_{\ell_2} = \sum_{n=1}^\infty (T e_n, e_n)_{\ell_2} = \Tr T < \infty$.
Hence $P(T) \in \cb_1(\ell_2)$ and $\Tr P(T) = \Tr T$.
By linearity $P(T) \in \cb_1(\ell_2)$ and $\Tr P(T) = \Tr T$ for all $T \in \cb_1(\ell_2)$.

If $S \in \cb_1(\ell_2)$ and $T \in \cl(\ell_2)$, then 
\begin{eqnarray*}
\langle S, P(T) \rangle_{\cb_1(\ell_2) \times \cl(\ell_2)}
& = & \Tr(S \, P(T))
= \sum_{n=1}^\infty (S \, P(T) e_n, e_n)_{\ell_2}  \\
& = & \sum_{n=1}^\infty (S ( (T e_n, e_n)_{\ell_2} \, e_n) , e_n)_{\ell_2}
= \sum_{n=1}^\infty (T e_n, e_n)_{\ell_2} \, (S e_n , e_n)_{\ell_2}
\end{eqnarray*}
and similarly 
\[
\langle P(S), T \rangle_{\cb_1(\ell_2) \times \cl(\ell_2)}
= \sum_{n=1}^\infty (T e_n, e_n)_{\ell_2} \, (S e_n , e_n)_{\ell_2}
.  \]
Hence 
\[
\langle S, P(T) \rangle_{\cb_1(\ell_2) \times \cl(\ell_2)}
= \langle P(S), T \rangle_{\cb_1(\ell_2) \times \cl(\ell_2)}
\]
for all $S \in \cb_1(\ell_2)$ and $T \in \cl(\ell_2)$.
If $P_{\cb_1(\ell_2)} \in \cl(\cb_1(\ell_2))$ denotes the restriction 
of $P$ to $\cb_1(\ell_2)$, then 
\[
\Big( P_{\cb_1(\ell_2)} \Big)^*
= P
\]
by the above.
So $P$ has a preadjoint.
\end{proof}

\begin{proof}[{\bf Proof of Theorem~\ref{tposvec606}.}]
`\ref{tposvec606-1}$\Rightarrow$\ref{tposvec606-3}'.
By Proposition~\ref{pposvec200.3} we may assume that $C_{kl}(x) = (C_{kl}(x)) \, \widecheck{\;}$
for all $k,l \in \{ 1,\ldots,d \} $ and $x \in \Omega$.
Let $\varphi,\psi \in V_\Omega$.
Define $T_{\varphi,\psi} \in \cl(\ell_2)$ by 
\[
T_{\varphi,\psi} 
= \sum_{k,l=1}^d \int_\Omega (\partial_l \varphi)(x) \, \overline{(\partial_k \psi)(x)} \, C_{kl}(x) \, dx
,  \]
where this time we use the weak$^*$-integral by considering $\cl(\ell_2) = (\cb_1(\ell_2))^*$.
As in the proof of Theorem~\ref{tposvec405} under Condition~\ref{tposvec405-2}
it follows that $T_{\varphi,\psi}$ is a multiplication operator.
Then 
Lemmas~\ref{lposvec602} and \ref{lposvec605} give
\[
T_{\varphi,\psi} 
= P(T_{\varphi,\psi})
= \sum_{k,l=1}^d \int_\Omega (\partial_l \varphi)(x) \, \overline{(\partial_k \psi)(x)} \, P( C_{kl}(x)) \, dx
.  \]
For all $n \in \Ni$ and $k,l \in \{ 1,\ldots,d \} $ define $c^{(n)}_{kl} \colon \Omega \to \Ri$ by
$c^{(n)}_{kl}(x) = (C_{kl}(x) e_n, e_n)_{\ell_2} = ((C_{kl}(x)) \, \widecheck{\;} e_n, e_n)_{\ell_2}$.
Then $c^{(n)}_{kl}$ is measurable and $|c^{(n)}_{kl}(x)| \leq M$ for all $x \in \Omega$.
Let $n \in \Ni$, $x \in \Omega$ and $\xi \in \Ci^d$.
Choose $f_k = \xi_k \, e_n$ for all $k \in \{ 1,\ldots,d \} $.
Then 
\begin{eqnarray*}
\RRe \sum_{k,l=1}^d c^{(n)}_{kl}(x) \, \xi_l \, \overline{\xi_k}
& = & \RRe \sum_{k,l=1}^d (C_{kl}(x) f_l, f_k)_{\ell_2}
\geq \mu \, \sum_{k=1}^d \|f_k\|_{\ell_2}^2
= \mu \, |\xi|^2
.
\end{eqnarray*}
Let $\varphi,\psi \in V_\Omega$ and $f,g \in \ell_2$.
Write $u = \varphi \otimes f$ and $v = \psi \otimes g$.
Then 
\begin{eqnarray*}
\gota(u,v)
& = & (T_{\varphi,\psi} f, g)_{\ell_2}  \\
& = & \sum_{k,l=1}^d \int_\Omega (\partial_l \varphi)(x) \, \overline{(\partial_k \psi)(x)} \,
     (P((C_{kl}(x)) \, \widecheck{\;} \, ) f, g)_{\ell_2} \, dx  \\
& = & \sum_{k,l=1}^d \int_\Omega \sum_{n=1}^\infty
      c^{(n)}_{kl}(x) \, (\partial_l u_n)(x) \, \overline{(\partial_k v_n)(x)} \, dx
.
\end{eqnarray*}
Then by density and continuity Condition~\ref{tposvec606-3} is valid.

`\ref{tposvec606-3}$\Rightarrow$\ref{tposvec606-2}'.
With the obvious definitions it only remains to show that 
$(S_t u)_n = S^{(n)}_t u_n$ for all $n \in \Ni$, $u \in \ell_2(L_2(\Omega))$
and $t > 0$.
Fix $n \in \Ni$.
Define $Q \colon \ell_2(L_2(\Omega)) \to L_2(\Omega)$ by $Q u = u_n$.
Then $Q^* \varphi = \varphi \otimes e_n$ for all $\varphi \in L_2(\Omega)$.
Therefore 
$Q V \subset V_\Omega$, $Q^* V_\Omega \subset V$ and 
$\gota_n(Q u, \varphi) = \gota(u,Q^* \varphi)$ for all 
$u \in V$ and $\varphi \in V_\Omega$.
Hence by \cite{AEK} Proposition~2.2(iii)$\Rightarrow$(i)
one deduces that $Q \, S_t = S^{(n)}_t \, Q$ for all $t > 0$.
So $(S_t u)_n = Q \, S_t u = S^{(n)}_t \, Q u = S^{(n)}_t u_n$
for all $u \in \ell_2(L_2(\Omega))$ and $t > 0$, as required.

`\ref{tposvec606-2}$\Rightarrow$\ref{tposvec606-1}'.
For all $n \in \Ni$ the semigroup $S^{(n)}$ is positive by 
\cite{Ouh5} Theorem~4.2.
Hence $S$ is a positive semigroup.
\end{proof}

\begin{proof}[{\bf Proof of Theorem~\ref{tposvec405} under assumption~\ref{tposvec405-5}}]
This follows immediately from Theorem~\ref{tposvec606}\ref{tposvec606-1}$\Rightarrow$\ref{tposvec606-3}.
\end{proof}

\begin{proof}[{\bf Proof of Theorem~\ref{tposvec110}.}]
First suppose that $Y = \Ni$.

`\ref{tposvec110-1}$\Rightarrow$\ref{tposvec110-2}'.
This follows from Theorem~\ref{tposvec606}\ref{tposvec606-1}$\Rightarrow$\ref{tposvec606-2}
and \cite{Ouh5} Theorem~4.2.

`\ref{tposvec110-2}$\Rightarrow$\ref{tposvec110-1}'.
This is trivial.

If $Y$ is a finite set, or $Y \subsetneq \Ni$, then the proof follows by obvious modifications.
\end{proof}

\begin{exam} \label{xposvec610}
We give an example of a finite Borel measure $\mu$ on a separable metrisable 
space~$X$ and a bounded function $C \colon X \to \cl(H)$ such that 
$x \mapsto (C(x) f, g)_H$ is Borel measurable for all $f,g \in H$,
but such that the Pettis integral does not exist.
Here $H$ is an infinite dimensional separable Hilbert space.

Let $\cu$ be a nonprincipal ultrafilter on $\Ni$.
We refer to \cite{HrbacekJech} Section~11.2 and Theorem~11.2.8
for background information on ultrafilters that is used in the following.
By Kadets and Leonov, \cite{KadetsLeonov} Corollary~2.16, Theorem~2.14, Lemma~2.13 and Theorem~2.2
there exist a subset $X$ of the Cantor space $ \{ 0,1 \} ^\Ni$,
a finite Borel measure $\mu$ on $X$ and for all $n \in \Ni$ 
a measurable function $\varphi_n \colon X \to \Ri$ such that 
\begin{itemize}
\item
$0 \leq \varphi_n(x) \leq 1$ for all $x \in X$ and $n \in \Ni$,
\item
$\lim\limits_\cu \varphi_n(x) = 0$ for all $x \in X$ and 
\item
$\lim\limits_\cu \int_X \varphi_n \, d\mu = \alpha > 0$.
\end{itemize}
This example is constructed to show that the dominated convergence theorem does not hold 
for convergence along a free ultrafilter on $\Ni$.

Now let $ \{ e_n : n \in \Ni \} $ be an orthonormal basis for $H$ and define 
$C \colon X \to \cl(H)$ by 
\[
C(x) f
= \sum_{n=1}^\infty \varphi_n(x) \, (f, e_n)_H \, e_n
.  \]
Let $f,g \in H$.
Then $(C(x) f,g)_H = \sum_{n=1}^\infty \varphi_n(x) \, (f,e_n)_H \, (e_n,g)_H$ 
for all $x \in X$, so $x \mapsto (C(x) f,g)_H$ is measurable.
Further $\|C(x)\|_{\cl(H)} \leq 1$ for all $x \in X$.
Define the operator $T \in \cl(H)$ by 
\[
(T f, g)_H 
= \int_X (C(x) f, g) \, d\mu(x)
\]
for all $f,g \in H$.
Next define $\Phi \in (\cl(H))^*$ by 
\[
\langle \Phi, S \rangle_{(\cl(H))^* \times \cl(H)}
= \lim\limits_\cu (S e_n, e_n)_H
.  \]
Observe that $(T e_n, e_n)_H = \int_X \varphi_n \, d\mu$ for all $n \in \Ni$.
Therefore $\langle \Phi, T \rangle_{(\cl(H))^* \times \cl(H)} = \alpha \neq 0$.

On the other hand, 
\[
\langle \Phi, C(x) \rangle_{(\cl(H))^* \times \cl(H)}
= \lim\limits_\cu \varphi_n(x) 
= 0
\]
for all $x \in X$.
Hence 
\[
\langle \Phi, T \rangle_{(\cl(H))^* \times \cl(H)}
\neq \int_X \langle \Phi, C(x) \rangle_{(\cl(H))^* \times \cl(H)} \, d\mu(x)
.  \]
This implies that $C$ is not Pettis integrable.
\end{exam}

\subsection*{Acknowledgements}
The second-named author is most grateful for the hospitality extended
to him during fruitful stays at Ulm University.
He wishes to thank Ulm University for financial support.

\end{document}